\documentclass[10pt]{amsart}

\usepackage{amsfonts}
\usepackage{amssymb}
\usepackage{amsmath, amsthm, latexsym}
\usepackage[all]{xy}
\usepackage{hyperref}
\usepackage{graphicx, color}
\usepackage[margin=1.5in]{geometry}

\def \Z{\mathbb{Z}}
\def \R{\mathbb{R}}

\def \Q{\mathbb{Q}}
\def \A{\mathcal{A}}

\def \E{\mathcal{E}}

\def \v{{\bf v}}
\def \fv{\mathfrak{v}}
\def \k{{\bf k}}

\def \w{{\bf w}}
\def \B{\mathfrak{B}}
\def \L{\mathcal{L}}
\def \P{\mathcal{P}}
\def \h{\fv}
\def \O{\mathcal{O}}

\def \trop{\operatorname{trop}}
\def \Trop{\operatorname{Trop}}

\def \Spec{\operatorname{Spec}}

\def \SL{\operatorname{SL}}
\def \GL{\operatorname{GL}}
\def \PGL{\operatorname{PGL}}

\def \Sym{\operatorname{Sym}}

\def \PL{\operatorname{PL}}
\def \CPL{\operatorname{CPL}}

\def \rank{\operatorname{rank}}
\def \val{\operatorname{val}}
\def \Span{\operatorname{span}}

\theoremstyle{plain}
\newtheorem{theorem}{Theorem}[section]
\newtheorem{lemma}[theorem]{Lemma}
\newtheorem{proposition}[theorem]{Proposition}
\newtheorem{corollary}[theorem]{Corollary}

\newtheorem{THM}{Theorem}

\newtheorem{PROP}[THM]{Propposition}

\theoremstyle{definition}
\newtheorem{example}[theorem]{Example}
\newtheorem{definition}[theorem]{Definition}
\newtheorem{remark}[theorem]{Remark}

\newtheorem*{DEFIN}{Definition}
\newtheorem*{REM}{Remark}

\newtheorem*{QUEST}{Question}

\begin{document}
\title{Toric vector bundles, valuations and tropical geometry}

\author{Kiumars Kaveh}
\address{Department of Mathematics, University of Pittsburgh,
Pittsburgh, PA, USA.}
\email{kaveh@pitt.edu}

\author{Christopher Manon}
\address{Department of Mathematics, University of Kentucky, Lexington, KY, USA}
\email{Christopher.Manon@uky.edu}

\date{\today}
\thanks{The first author is partially supported by a National Science Foundation Grant (Grant ID: DMS-1601303) and a Simons Collaboration Grant for Mathematicians.}
\thanks{The second author is partially supported by National Science Foundation Grant DMS-2101911 and a Simons Collaboration Grant (award number 587209).}
\subjclass[2010]{}
\begin{abstract}
A toric vector bundle $\E$ is a torus equivariant vector bundle on a toric variety. 
We give a valuation theoretic and tropical point of view on toric vector bundles. We present three (equivalent) classifications of toric vector bundles, which should be regarded as repackagings of the Klyachko data of compatible $\Z$-filtrations of a toric vector bundle: (1) as piecewise linear maps to space of $\Z$-valued valuations, (2) as valuations with values in the semifield of piecewise linear functions, and (3) as points in tropical linear ideals over the semifield of piecewise linear functions. 
Moreover, we interpret the known criteria for ampleness and global generation of $\E$ as convexity conditions on its piecewise linear map in (1). 
Finally, using (2) we associate to $\E$ a collection of polytopes indexed by elements of a certain (representable) matroid encoding the dimensions of weight spaces of global sections of $\E$. This recovers and extends the Di Rocco-Jabbusch-Smith matriod and parliament of polytopes of $\E$.  
This is a follow up paper to \cite{KM-tpb}.
\end{abstract}

\maketitle

\tableofcontents

\section*{Introduction}
The purpose of this paper is twofold: (1) cast some known results in the theory of toric vector bundles in the language of piecewise linear maps and buildings, and (2) make a connection between the theory of toric vector bundles and valuation theory and tropical geometry over the piecewise linear semifield.  

Throughout $\k$ denotes the base field which we take to be of characteristic $0$. We let $T \cong (\k^*)^n$ be an $n$-dimensional (split) algebraic torus over $\k$. We let $M$ and $N$ denote the lattices of characters and cocharacters of $T$ respectively, and we put $N_\R = N \otimes_\Z \R$ and $M_\R = M \otimes_\Z \R$. 
Let $\Sigma$ be a fan in $N_\R \cong \R^n$ with $X_\Sigma$ its associated toric variety. Recall that a toric variety is a normal variety equipped with an action of algebraic torus $T$ such that $T$ has an open orbit isomorphic to $T$ itself. For the rest of the paper, we fix a point $x_0$ in the open torus orbit in $X_\Sigma$. 
A toric vector bundle $\E$ on $X_\Sigma$ is a vector bundle with a $T$-linearization, namely a linear action of $T$ on $\E$ that lifts the action of $T$ on $X_\Sigma$.

It is well-known that toric line bundles on $X_\Sigma$ are in one-to-one correspondence with functions $\phi: |\Sigma| \to \R$ that are piecewise linear with respect to $\Sigma$ and are integral, i.e. map $|\Sigma| \cap \Z^n$ to $\Z$ (here $|\Sigma|$ denotes the support of $\Sigma$, the union of all cones in $\Sigma$). The first classification of toric vector bundles goes back to \cite{Kaneyama} and is in terms of certain cocycles. Alternatively, Klyachko gave a classification in terms of certain compatible filtrations on a finite dimensional vector space \cite{Klyachko}. The Klyachko classification has been used in much of the work in the literature, for example \cite{Payne-moduli}, \cite{HMP}, \cite{DJS} and \cite{KM-PL}. Below is an overview of the content of the paper.


\smallskip
\noindent \textbf{Toric vector bundles as piecewise linear maps to (cones over) Tits buildings.}
Let $\E$ be a toric vector bundle over $X_\Sigma$ with $E = \E_{x_0}$ the fiber over $x_0$.
The Klyachko classification is in terms of data of compatible (decreasing) $\Z$-filtrations on the $\k$-vector space $E$ (see Section \ref{sec-prelim-toric-vb}). In \cite{KM-tpb}, the authors interpret the Klyachko data of filtrations as a \emph{piecewise linear map} $\Phi$ from the fan $\Sigma$ to $\tilde{\B}(\GL(E))$, the \emph{(cone over the) Tits building of the general linear group $\GL(E)$}. This point of view is then used to give a Klyachko type classification of torus equivariant principal $G$-bundles on toric varieties, for any reductive algebraic group $G$, in terms of piecewise linear maps to (the cone over) the Tits building of $G$ (see \cite[Theorem 1]{KM-tpb}).

The (cone over the) Tits building of $\GL(E)$ can be realized as the space of $\R$-valued valuations on the vector space $E$. We simply denote this space by $\tilde{\B}(E)$. In Section \ref{sec-prelim} we review some background material about valuations and buildings. In Section \ref{subsec-PL-maps}, as a special case of the notion of piecewise linear map in \cite{KM-tpb}, we reformulate the Klyachko data of a toric vector bundle as a piecewise linear map $\Phi: |\Sigma| \to \tilde{\B}(E)$. 
We point out that the idea of regarding the Klyachko data of compatible filtrations as a piecewise linear map is not quite new and goes back to \cite{Payne-cover}. Nevertheless, we find this packaging of Klyachko data, coupled with insights from the theory of buildings, quite helpful. For example, beside the classification of toric principal bundles in \cite{KM-tpb}, it provides the right gadget to classify toric vector bundles over toric schemes over a discrete valuation ring (see \cite{KMT}).


\smallskip
\noindent \textbf{Positivity of toric vector bundles in terms of piecewise linear maps.}
In \cite{HMP} and \cite{DJS}, criteria are given for nefness/ampleness and global generation of toric vector bundles in terms of their Klyachko data. In Section \ref{subsec-tvb-positivity}, we interpret these criteria in terms of two convexity notions for piecewise linear maps to (cones over) the Tits buildings which we call \emph{buildingwise convexity} and \emph{fanwise convexity}. 
\begin{THM}
Let $\E$ be a toric vector bundle on a complete toric variety $X_\Sigma$ with corresponding piecewise linear map $\Phi: |\Sigma| \to \tilde{\B}(E)$. 
\begin{itemize}
\item[(1)] $\E$ is nef (respectively ample) if and only if $\Phi$ is buildingwise convex (respectively strictly buildingwise convex).
\item[(2)] $\E$ is globally generated if and only if $\Phi$ is fanwise convex.
\end{itemize}
\end{THM}
While a toric line bundle is nef if and only if it is globally generated, the notions of nef and globally generated are different for toric vector bundles (see \cite[Example 5.3]{DJS}). The above theorem shows that this is reflected in the fact that the notions of buildingwise convex and fanwise convex are different.

We hope that the convexity conditions on $\Phi_\E$ for ampleness and global generation of $\E$ will be useful in attacking the Fujita conjecture for projectivized toric vector bundles. 

\begin{QUEST}
Can we formulate the fanwise and buildingwise convexity of $\Phi: |\Sigma| \to \tilde{\B}(E)$ in terms of convexity notions in the building?
\end{QUEST}

\smallskip 
\noindent \textbf{Toric vector bundles as piecewise linear valuations.}
In commutative algebra, one usually considers valuations with values in an ordered abelian group. This definition, word by word, extends to a valuation with values in an \emph{idempotent semifield} (see Definitions \ref{def-idempotent-semifield},  \ref{def-valuationO} as well as \cite{Giansiracusa}). 

In our case, we are interested in the semifield of (integral) piecewise linear functions. 
Let $\PL(N, \Z)$ denote the set of all piecewise linear functions, with respect to some complete fan, on the vector space $N_\R$ which have integer values on the lattice $N$. Moreover, we add a unique ``infinity element'' $\infty$ to $\PL(N, \Z)$ which is greater than any other element. We regard it as the function which assigns value infinity to all points in $N_\R$. The set $\PL(N, \Z)$ is an idempotent semifield with operations of taking minimum and addition of functions. We call a valuation on a vector space and with values in $\PL(N, \Z)$ a \emph{piecewise linear valuation}.
More precisely we have the following definition:
\begin{DEFIN}[Piecewise linear valuation]
Let $E$ be a finite dimensional $\k$-vector space. A map $\fv: E \to\PL(N, \Z)$ is a \emph{piecewise linear valuation} if:
\begin{itemize}
\item[(a)] For any $e \in E$ and any $0 \neq c \in \k$ we have $\fv(c e) = \fv(e)$.
\item[(b)] For any $e_1, e_2 \in E$ we have
$\fv(e_1+e_2) \geq \min(\fv(e_1), \fv(e_2))$. 
\item[(c)] $\fv(e) = \infty$ if and only if $e = 0$.
\end{itemize}
If the image of $\fv$ is finite we call $\fv$ a \emph{finite piecewise linear valuation}.
\end{DEFIN}

\begin{REM}[Piecewise linear valuations on algebras]
Let $A$ be a $\k$-algebra and domain. A map $\fv: A \to \PL(N, \Z)$ is called a \emph{piecewise linear valuation} on $A$ if it satisfies (a)-(c) and moreover $\fv(f_1f_2) = \fv(f_1)+\fv(f_2)$, for all $f_1, f_2 \in A$. If $A$ is graded, we say that $\fv$ is \emph{homogeneous} if $\fv(f)$ is equal to the minimum of values of $\fv$ on the homogeneous components of $f$. It is easy to show that a piecewise linear valuation on $E$ extends uniquely to a homogeneous piecewise linear valuation on the symmetric algebra $\Sym(E)$. In fact, valuations on $E$ are in one-to-one correspondence with homogeneous valuations on $\Sym(E)$.
\end{REM}

Let $\Sigma_1$, $\Sigma_2$ be fans with the same support. We say that toric vector bundles $\E_i$ over  $X_{\Sigma_i}$, $i=1, 2$, are \emph{equivalent} if there is a common refinement $\Sigma$ of the $\Sigma_i$ such that $\pi_1^*(\E_1) \cong \pi_2^*(\E_2)$, where $\pi_i: X_\Sigma \to X_{\Sigma_i}$ is the toric blow-up corresponding to refinement of $\Sigma_i$ to $\Sigma$. 
\begin{THM}[Toric vector bundles as piecewise linear valuations]   \label{th-intro-tvb-preval}
The equivalence classes of toric vector bundles $\E$ on complete $T$-toric varieties $X_\Sigma$ , with $E$ as the fiber over the distinguished point $x_0$, with respect to the above equivalence (pull back by toric blowups), are in one-to-one correspondence with finite piecewise linear valuations $\fv: E \to \PL(N, \Z)$.  
\end{THM}

A key ingredient in the proof of Theorem \ref{th-intro-tvb-preval} is Theorem \ref{th-plm-vs-preval-plf} which gives a correspondence between integral piecewise linear maps $\Phi$ to $\tilde{\B}(E)$ and piecewise linear valuations $\fv: E \to \PL(N, \Z)$. 

An interesting feature of realizing a toric vector bundle $\E$ as a piecewise linear valuation is that it allows one to immediately recover and extend the \emph{matroid} and \emph{parliament of polytopes} of $\E$, introduced in \cite{DJS}. The parliament of polytopes $P_\E$ is a finite collection of convex polytopes in $M_\R$ and labeled by elements of a (representable) matroid $\mathcal{M}$ in $E$. It plays the role of the moment polytope/Newton polytope of a $T$-linearized line bundle on a toric variety. In particular, the dimensions of weight spaces of global sections of $\E$ can be read off from its parliament (see \cite[Proposition 1.1]{DJS}). The matroid and parliament of polytopes of $\E$ are constructed from its \emph{Klyachko arrangement}. The Klyachko arrangement of a toric vector bundle is the subspace arrangement obtained by intersecting any number of subspaces appearing in its Klyachko filtrations.

Below we explain how the notion of parliament of polytopes is related to the piecewise valuation $\fv$ of a toric vector bundle. Let $\fv: E \to \PL(N, \Z)$ be a valuation. Let $S \subset \PL(N, \Z)$ be a join-subsemilattice, that is, $S$ is closed under taking maximum. To $S$ there corresponds a subspace arrangement 
$\mathcal{A}_{\fv, S}$ consisting of all subspaces $E_{\geq \phi} = \{ e \in E \mid \fv(e) \geq \phi \}$, for all $\phi \in S$. Since $\fv$ is closed under taking maximum, $\mathcal{A}_{\fv, S}$ is closed under intersection. 
We also recall that to every piecewise linear function $\phi: N_\R \to \R$ there corresponds a polytope (possibly empty): 
$$P_\phi = \{ y \in M_\R \mid \langle x, y \rangle \leq \phi(x),~\forall x \in N_\R \}.$$
We let $M(\fv, S) \subset E$ denote the matroid of the subspace arrangement $\mathcal{A}_{\fv, S}$ (see Theorem \ref{th-matroid-of-arrangement}). We let $$P(\fv, S) = \{P_{\fv(e)} \mid e \in M(\fv, S) \},$$ and call it the \emph{parliament of polytopes} of $(\fv, S)$.

Let $\E$ be a toric vector bundle with corresponding finite piecewise linear valuation $\fv: E \to \PL(N, \Z)$. The following shows that the dimensions of weight spaces of global sections of $\E$ can be recovered from the parliament $P(\fv, S)$ (see Theorem \ref{th-parliament-global-sec}). It generalizes \cite[Proposition 1.1]{DJS}:
\begin{PROP}
Let $\E$ be a toric vector bundle over a complete toric variety $X_\Sigma$ with corresponding piecewise linear valuation $\fv: E \to \PL(N, \Z)$. Let $S \subset \PL(N, \Z)$ be a subset that contains the character lattice $M$ and is closed under taking maximum. Then for any $u \in M$ we have:  
$$\dim(H^0(X_\Sigma, \E)_u) = \rank\{ e \in M(\fv, S) \mid u \in P_{\fv(e)}\}.$$
\end{PROP}
More generally, we can consider the larger semifield $\widehat{\PL}(N, \overline{\Z})$ consisting of functions $\phi: N \to \overline{\Z} = \Z \cup \{\infty\}$ which are homogeneous of degree $1$, that is, $\phi(c x) = \phi(x)$, for all $c \in \Z$. 
The Di Rocco-Jabusch-Smith parliament of polytopes of $\E$ can be recovered from the above construction for a certain subset $S_\Sigma \subset \widehat{\PL}(N, \overline{\Z})$ (see Proposition \ref{prop-arrangement-intersect-E-rho}).
It is well-known that the space of polytopes has a natural semiring structure given by the convex hull of union and the Minkowski sum. This semiring can be identified with the semiring of concave piecewise linear functions (see Section \ref{subsec-preval-plf}). In light of Theorem \ref{th-intro-tvb-preval} we ask the following question.
\begin{QUEST}
What toric vector bundles correspond to piecewise linear valuations on $E$ with values in the semialgebra of polytopes (equvialently, concave piecewise linear functions)? (See Section \ref{subsec-preval-plf}.)
\end{QUEST}


\begin{REM}
Theorem \ref{th-intro-tvb-preval} is the opening act for the companion paper \cite{KM-PL} where the idea of classification of toric vector bundles by piecewise linear valuations is far extended to \emph{toric flat families}. One of the main results states that torus equivariant flat families $\pi: \mathcal{X} \to X_\Sigma$ with generic fiber $Y=\Spec(A)$ are classified by \emph{piecewise linear valuations} on $A$ (\cite[Theorem 1.2]{KM-PL}). This piecewise linear valuation perspective is then used to obtain results on finite generation of Cox rings of projectivized toric vector bundles (\cite[Section 6]{KM-PL}). This point of view also opens doors to the study of tropical geometry over the semifield of piecewise linear functions.    
\end{REM}

\smallskip
\noindent \textbf{Toric vector bundles as tropical points.}
Given an ideal $I$ in a polynomial ring $\k[x_1, \ldots, x_s]$ and an idempotent semifield $(\mathcal{O}, \oplus, \otimes)$, one defines the tropical variety $\Trop_\mathcal{O}(I) \subset \mathcal{O}^s$ (see Section \ref{sec-tvbs-trop-points}). When $(\mathcal{O}, \oplus, \otimes) = (\overline{\R}, \min, +)$ is the tropical semifield, the tropical variety is denoted by $\Trop(I)$. It is the support of a fan in $\R^s$ and contains important information about the variety defined by $I$. Its study is the subject of tropical geometry.  

\begin{QUEST}
What geometric information are encoded in tropical varieties over the piecewise linear semifield $\PL(N, \Z)$? 
\end{QUEST}
In Section \ref{sec-tvbs-trop-points}, we take a step towards answering this question. We see that points in the tropical varieties of linear ideals over $\PL(N, \Z)$ are in correspondence with toric vector bundles! More precisely, we have the following (see Theorem \ref{th-val-vs-trop-pt} and Proposition \ref{prop-tropdata} for more details):
\begin{THM}[Toric vector bundles as tropical points]
Let $B=\{b_1, \ldots, b_s\}$ be a spanning set for the vector space $E$. Let $L \subset \k[x_1, \ldots, x_s]$ be the ideal generated by the linear relations among the $b_i$.
\begin{itemize}
\item[(i)] A tuple $(\phi_1, \ldots, \phi_s) \in (\PL(N, \Z))^s$ lies in $\Trop_{\PL(N, \Z)}(L)$, if and only if there is a (unique) valuation $\fv:E \to \PL(N, \Z)$ with $\fv(b_i) = \phi_i$, for all $i$.  
\item[(ii)] In light of Theorem \ref{th-intro-tvb-preval}, (i) implies that the points in $\Trop_{\PL(N, \Z)}(L)$ correspond to toric vector bundles (up to pull-back by toric blowups).
\end{itemize}    
\end{THM}

More systematic study of geometric data encoded by tropical points over the piecewise linear semifield $\PL(N, \Z)$ has been initiated in \cite{KM-PL}.


\bigskip
\noindent{\bf Acknowledgement.} We would like to thank Sam Payne, Greg Smith, Kelly Jabbusch, Sandra Di Rocco, Roman Fedorov, Bogdan Ion for useful conversations and email correspondence. The first author is partially supported by National Science Foundation Grants (DMS-1601303 and DMS-2101843) and a Simons Collaboration Grant (award number 714052). The second author is partially supported by National Science Foundation Grants (DMS-1500966 and DMS-2101911) and a Simons Collaboration Grant (award number 587209). 


\bigskip

\noindent{\bf Notation.} Throughout the paper we will use the
following notation: 
\begin{itemize}
\item $\k$ is the base field which we take to be of characteristic $0$. 
\item $E \cong \k^r$ is a finite dimensional $\k$-vector space.
\item $\Delta(E)$ is the Tits building of $\GL(E)$. Its simplices correspond to the flags of subspaces in $E$ (in other words, the parabolic subgroups in $\GL(E)$). Apartments correspond to choices of frames in $E$, that is, a direct sum decompositions of $E$ into $1$-dimensional subspaces (in other words, the maximal tori in $\GL(E)$). 
\item $\B(E)$ denotes the geometric realization of the Tits building of $\GL(E)$. It is an infinite union of $(r-2)$-dimensional spheres, one sphere for each apartment in $\Delta(E)$. Each sphere is partitioned into subsets homeomophic to standard simplices corresponding to simplices in the apartment. The spheres are glued together along common simplices in the corresponding apartments (Section \ref{subsec-val-Tits}).
\item $\tilde{\B}(E)$ denotes the set of all valuations $v: E \to \overline{\R} = \R \cup \{\infty\}$ (Definition \ref{def-val}). We denote the set of integral valuations, i.e. $v: E \to \overline{\Z}$, by $\tilde{\B}_\Z(E)$. The Tits building $\B(E)$ can be obtained as set of equivalence classes of valuation. We say that $v_1 \sim v_2$ if $v_2 = mv_1 + c$ where $m, c \in \R$ with $m > 0$. We refer to $\tilde{\B}(E)$ as the \emph{cone over the Tits building} or the \emph{extended Tits building}  of $\GL(E)$ (see Section \ref{subsec-vs-val} and Section \ref{subsec-val-Tits}). 
\item $T \cong \mathbb{G}_m^n$ denotes a (split) algebraic torus with $M$ and $N$ its character and cocharacter lattices respectively. In general, $M$ and $N$ denote rank $n$ free abelian groups dual to each other. We denote the pairing between them by $\langle \cdot, \cdot \rangle: M \times N \to \Z$. We let $M_\R = M \otimes \R$ and $N_\R = N \otimes \R$ be the corresponding $\R$-vector spaces. 
\item $U_\sigma$ is the affine toric variety corresponding to a (strictly convex rational polyhedral) cone $\sigma \subset N_\R$.
\item $\Sigma$ is a fan in $N_\R$ with corresponding toric variety $X_\Sigma$. We denote the support of $\Sigma$, i.e. the union of cones in it, by $|\Sigma|$. \item $\Phi: |\Sigma| \to \tilde{\B}(E)$ is a piecewise linear map to the space of valuations on a $\k$-vector space $E$ (Section \ref{subsec-PL-maps}).
\item $\PL(N_\R, \R)$ and $\CPL(N_\R, \R)$, the sets of piecewise linear functions and concave piecewise linear functions on the $\R$-vector space $N_\R$ respectively. We denote the set of piecewise linear functions (respectively concave piecewise linear functions) that attain integer values on $N$ by $\PL(N, \Z)$ (respectively $\CPL(N, \Z)$). Finally $\PL(\Sigma, \R)$ (respectively $\PL(\Sigma, \Z)$) denotes the subset of piecewise linear functions (respectively integral piecewise linear functions) that are linear on cones in $\Sigma$ (Section \ref{subsec-preval-plf}). 
\item $\P(M_\R)$, the set of polytopes in the $\R$-vector space $M_\R$. We denote the set of lattice polytopes in $M_\R$ by $\P(M)$ (Section \ref{subsec-preval-plf}).
\end{itemize}

\section{Preliminaries}   \label{sec-prelim}
\subsection{Klyachko classification of toric vector bundles}   \label{sec-prelim-toric-vb}
Let $T \cong \mathbb{G}_m^n$ denote an $n$-dimensional (split) algebraic torus over a field $\k$. We let $M$ and $N$ denote its character and cocharacter lattices respectively. We also denote by $M_\R$ and $N_\R$ the $\R$-vector spaces spanned by $M$ and $N$. For cone $\sigma \in N_\R$ let $M_\sigma$ be the quotient lattice:
$$M_\sigma = M / (\sigma^\perp \cap M).$$
Let $\Sigma$ be a (finite rational polyhedral) fan in $N_\R$ and let $X_\Sigma$ be the corresponding toric variety. Also $U_\sigma$ denotes the invariant affine open subset in $X_\Sigma$ corresponding to a cone $\sigma \in \Sigma$. We denote the support of $\Sigma$, that is the union of all the cones in $\Sigma$, by $|\Sigma|$. For each $i$, $\Sigma(i)$ denotes the subset of $i$-dimensional cones in $\Sigma$. In particular, $\Sigma(1)$ is the set of rays in $\Sigma$. For each ray $\rho \in \Sigma(1)$ we let $\v_\rho$ be the primitive vector along $\rho$, i.e. $\v_\rho$ is the unique vector on $\rho$ whose integral length is equal to $1$.

We say that $\E$ is a {\it toric vector bundle} on $X_\Sigma$ if $\E$ is a vector bundle on $X_\Sigma$ equipped with a $T$-linearization. This means that there is an action of $T$ on $\E$ which lifts the $T$-action on $X_\Sigma$ such that the action map $\E_x \to \E_{t \cdot x}$ for any $t \in T$, $x \in X_\Sigma$ is linear. 

We fix a point $x_0 \in X_0 \subset X_\Sigma$ in the dense orbit $X_0$. We often identify $X_0$ with $T$ and think of $x_0$ as the identity element in $T$. We let $E = \E_{x_0}$ denote the fiber of $\E$ over $x_0$. It is an $r$-dimensional vector space where $r = \rank(\E)$. 

For each cone $\sigma \in \Sigma$, with invariant open subset $U_\sigma \subset X_\Sigma$, the space of sections $\Gamma(U_\sigma, \E)$ is a $T$-module. We let  $\Gamma(U_\sigma, \E)_u \subseteq\Gamma(U_\sigma, \E)$ be the weight space corresponding to a weight $u \in M$. One has the weight decomposition: 

$$\Gamma(U_\sigma, \E) = \bigoplus_{u \in M} \Gamma(U_\sigma, \E)_u.$$ Every section in $\Gamma(U_\sigma, \E)_u$ is determined by its value at $x_0$.  Thus, by restricting sections to $E = \E_{x_0}$, we get an embedding $\Gamma(U_\sigma, \E)_u \hookrightarrow E$. Let us denote the image of $\Gamma(U_\sigma, \E)_u$ in $E$ by $E_u^\sigma$. Note that if $u' \in \sigma^\vee \cap M$ then multiplication by the character $\chi^{u'}$ gives an injection $\Gamma(U_\sigma, \E)_u \hookrightarrow \Gamma(U_\sigma, \E)_{u-u'}$. Moreover, the multiplication map by $\chi^{u'}$ commutes with the evaluation at $x_0$ and hence induces an inclusion $E_u^\sigma \subset E_{u-u'}^\sigma$. If $u' \in \sigma^\perp$ then these maps are isomorphisms and thus $E_u^\sigma$ depends only on the class $[u] \in M_\sigma = M/(\sigma^\perp \cap M)$. For a ray $\rho \in \Sigma(1)$ we write $$E^\rho_i = E_u^\rho,$$ for any $u \in M$ with $\langle u, \v_\rho \rangle = i$ (all such $u$ define the same class in $M_\rho$). Equivalently, one can define $E^\rho_u$ as follows (see \cite[\S 0.1]{Klyachko}). Pick a point $x_\rho$ in the orbit $O_\rho$ and let:
$$E^\rho_u = \{ e \in E \mid \lim_{t \cdot x_0 \to x_\rho} \chi^u(t)^{-1}(t \cdot e) \text{ exists in } \E \},$$
where $t$ varies in $T$ in such a way that $t \cdot x_0$ approaches $x_\rho$. 
We thus have a decreasing filtration of $E$:
\begin{equation}  \label{equ-filt-E-rho}
\cdots \supset E^\rho_{i-1} \supset E^\rho_i \supset E^\rho_{i+1} \supset \cdots
\end{equation}

An important step in the classification of toric vector bundles is that a toric vector bundle over an affine toric variety is {\it equivariantly trivial}. That is, it decomposes $T$-equivariantly as a sum of trivial line bundles. Let $\sigma$ be a strictly convex rational polyhedral cone with corresponding affine toric variety $U_\sigma$. Given $u \in M$, let $\L_u$ be the trivial line bundle $U_\sigma \times \mathbb{A}^1$ on $U_\sigma$ where $T$ acts on $\mathbb{A}^1$ via the character $u$. One observes that in fact the ($T$-equivariant isomorphism class of) toric line bundle $\L_u$ only depends on the class $[u] \in M_\sigma$. Hence we also denote this line bundle by $\L_{[u]}$. One has the following (see \cite[Proposition 2.1.1]{Klyachko}):

\begin{proposition}   \label{prop-toric-vb-over-affine-equiv-trivial}
Let $\E$ be a toric vector bundle of rank $r$ on an affine toric variety $U_\sigma$. Then $\E$ splits equivariantly into a sum of line bundles: 
$$\E = \bigoplus_{i=1}^r \L_{[u_i]},$$
where $[u_i] \in M_\sigma$.
\end{proposition}
We denote the multiset $\{ [u_1], \ldots, [u_r]\} \subset M_\sigma$ by $u(\sigma)$. The above shows that, for each $\sigma \in \Sigma$, the filtrations $(E^\rho_i)_{i \in \Z}$, $\rho \in \Sigma(1)$, satisfy the following compatibility condition: 
There is a decomposition of $E$ into a direct sum of $1$-dimensional subspaces indexed by a finite subset $u(\sigma) \subset M_\sigma$:
$$E = \bigoplus_{[u] \in u(\sigma)} L_{[u]},$$
such that for any ray $\rho \in \sigma(1)$ we have:
\begin{equation}  \label{equ-Klyachko-comp-condition}
E^\rho_i = \sum_{\langle u, \v_\rho \rangle \geq i}  L_{[u]}
\end{equation}

\begin{definition}[Compatible collection of filtrations]
We call a collection of decreasing $\Z$-filtrations $\{(E_i^\rho)_{i \in \Z} \mid \rho \in \Sigma(1) \}$ satisfying condition \eqref{equ-Klyachko-comp-condition} a \emph{compatible collection of filtrations}. 
(Moreover, for each $\rho$, we assume $\bigcap_{i \in \Z} E^\rho_i = \{0\}$ and $\bigcup_{i \in \Z} E^\rho_i = E$.)
\end{definition}

Let $E$, $E'$ be finite dimensional $\k$-vector spaces. Let $\{(E_i^\rho)_{i \in \Z} \mid \rho \in \Sigma(1) \}$ (respectively $\{({E'}_i^\rho)_{i \in \Z} \mid \rho \in \Sigma(1) \}$) be compatible collections of filtrations on $E$ (respectively $E'$). We say that a linear map $F: E \to E'$ is a \emph{morphism} from $\{(E_i^\rho)_{i \in \Z} \mid \rho \in \Sigma(1) \}$ to $\{({E'}_i^\rho)_{i \in \Z} \mid \rho \in \Sigma(1) \}$ if for every $\rho \in \Sigma(1)$ and $i \in \Z$ we have $F(E_i^\rho) \subset {E'}_i^\rho$. 
With this notion of morphism, for a fixed fan $\Sigma$, the compatible collections of filtrations on finite dimensional $\k$-vector spaces form a category.

The following is Klaychko's theorem on the classification of toric vector bundles (\cite[Theorem 2.2.1]{Klyachko}). 
\begin{theorem}[Klyachko]   \label{th-Klyachko}
The category of toric vector bundles on $X_\Sigma$ is naturally equivalent to the category of compatible filtrations on finite dimensional $\k$-vector spaces.
\end{theorem}


\subsection{Vector space valuations with values in real numbers} \label{subsec-vs-val}
In this section we consider the notion of a valuation on a vector space $E$ with values in $\R$. In Section \ref{subsec-PL-maps} we interpret the Klyachko data of compatible filtrations, for a toric vector bundle $\E$ on $X_\Sigma$ as an (integral) \emph{piecewise linear map} $\Phi$ from $|\Sigma|$ to the space $\tilde{\B}(E)$ of all valuations on $E$. 
We remark that the piecewise linear map $\Phi$ is essentially contained in Payne’s
observation in \cite{Payne-cover} that the Klyachko data of a toric vector bundle can be used to construct a filtration-valued function on $|\Sigma|$.
This is also a special case of the main result in \cite{KM-tpb} where torus equivariant principal $G$-bundles over $X_\Sigma$, where $G$ is a reductive algebraic group, are classified in terms of \emph{piecewise linear maps to the (extended) Tits building of $G$}. 

\begin{definition}[Vector space valuation]   \label{def-val}
Let $E$ be a finite dimensional $\k$-vector space.
We call a function $v: E \to \overline{\R} = \R \cup \{\infty \}$ a {\it vector space valuation} (or a \emph{valuation} for short) if the following hold:
\begin{itemize}
\item[(1)] For all $e \in E$ and $0 \neq c \in \k$ we have $v(ce) = v(e)$. 
\item[(2)](Non-Archimedean property) For all $e_1, e_2 \in E$, $v(e_1+e_2) \geq \min\{v(e_1), v(e_2)\}$.
\item[(3)] $v(e)=\infty$ if and only if $e = 0$.
\end{itemize}
We call a valuation $v$ {\it integral} if it attains only integer values, i.e. $v: E \to \overline{\Z}$. 
\end{definition}

\begin{remark}
Here are two remarks about the term \emph{valuation}:
\begin{itemize}
\item[(i)] In commutative algebra the term valuation usually refers to a valuation on a ring or algebra. Throughout most of this paper, we will use the term valuation to mean a valuation on a vector space. Later in Section \ref{subsec-semilattice-val} we define the more general notion of a \emph{semilattice valuation} that is a valuation with values in a semilattice (in place of $\R$). 

\item[(ii)] In \cite[Section 2.1]{KKh-Annals} (and some other papers) the term \emph{prevaluation} is used for a valuation on a vector space (to distinguish it from valuations on rings). 
\end{itemize}
\end{remark}

The {\it value set} $v(E)$ of a valuation $v$ is the image of $E \setminus \{0\}$ under $v$, i.e.
$$v(E) = \{ v(e) \mid 0 \neq e \in E\}.$$ It is easy to verify that $|v(E)| \leq \dim(E)$ and hence $v(E)$ is finite. Each integral valuation $v$ on $E$ gives rise to a filtration $E_{v, \bullet} = (E_{v \geq a})_{a \in \Z}$ on $E$ by vector subspaces defined by: 
$$E_{v \geq a} = \{ e \in E \mid v(e) \geq a\}.$$
If $v(E) = \{a_1 > \cdots > a_k\}$ then we have a flag: $$F_{v, \bullet} = (\{0\} \subsetneqq F_1 \subsetneqq \cdots \subsetneqq F_k=E),$$ where $F_i = E_{v \geq a_i}$. We note that the valuation $v$ is uniquely determined by the flag $F_{v, \bullet}$ and the $k$-tuple $(a_1 > \cdots > a_k)$.
Conversely, a decreasing filtration $E_\bullet = (E_a)_{a \in \Z}$ such that 
\begin{equation} \label{equ-flitration-conditions}
\bigcap_{a \in \Z} E_a = \{0\}, \text{ and } \bigcup_{a \in \Z} E_a = E, 
\end{equation}
defines a valuation $v_{E_\bullet}$ by:
$$v_{E_\bullet}(e) = \max\{ a \in \Z \mid e \in E_a\},$$ for all $e \in E$. 
The following is straightforward to verify. 
\begin{proposition}   \label{prop-preval-filtration}
The assignments $v \mapsto E_{v, \bullet}$ and $v \mapsto (F_{v, \bullet}, (a_1 >  \cdots > a_k))$ give one-to-one correspondences between the following sets: 
\begin{itemize}
\item[(i)] The set of integral valuations $v: E \to \overline{\Z}$. 
\item[(ii)] The set of decreasing  $\Z$-filtrations $E_\bullet$ on $E$ satisfying \eqref{equ-flitration-conditions}.
\item[(iii)] The set of flags $F_\bullet = (\{0\} \subsetneqq F_1 \subsetneqq \cdots \subsetneqq F_k = E)$ together with tuples of integers $(a_1 > \cdots > a_k)$. 
\end{itemize}
\end{proposition}
Recall that a frame $L = \{L_1, \ldots, L_r\}$ for $E$ is a collection of $1$-dimensional subspaces $L_i$ such that $E = \bigoplus_{i=1}^r L_i$. We say that a valuation $v$ is \emph{adapted} to a frame $L$ if every subspace $E_{v \geq a}$ is a sum of some of the $L_i$. This is equivalent to the following: For any $e \in E$ let us write $e = \sum_i e_i$ where $e_i \in L_i$. Then:
\begin{equation}  \label{equ-val-min}
v(e) = \min\{ v(e_i) \mid i=1, \ldots, r \}.
\end{equation}
If a valuation $v$ is adapted to a frame $L$, then $v$ is uniquely determined by the $r$-tuple $(v(L_1), \ldots, v(L_r))$. Conversely, any $r$-typle $(a_1, \ldots, a_r) \in \R^r$ determines a unique valuation $v$ adapted to $L$ by requiring that $v(e_i) = a_i$, for all $i=1, \ldots, r$ and $0 \neq e_i \in L_i$. In other words, $v$ is given by $v(e) = \min\{a_i \mid e_i \neq 0\}$.

\begin{definition}[Space of valuations]   \label{def-space-of-val}
We denote by $\tilde{\B}(E)$ the set of all $\R$-valued valuations $v: E \to \overline{\R}$. We also denote the set of all $\Z$-valued valuations on $E$ (that is, the set of integral valuations on $E$) by $\tilde{\B}_\Z(E)$. For a frame $L$, we denote the set of valuations adapted to $L$ by $\tilde{A}(L)$. Also we denote by $\tilde{A}_\Z(L)$ the set of $\Z$-valued valuations adapted to $L$. As discussed above, $\tilde{A}(L)$ (respectively $\tilde{A}_\Z(L)$) can be identified with $\R^r$ (respectively $\Z^r$).
\end{definition}

The space of valuations has a natural partial order.
\begin{definition}[Partial order on the space of valuations]   \label{def-po-valuation}
Let $v$, $w: E \to \overline{\R}$ be valuations on $E$. We say that $v \leq w$ if $v(e) \leq w(e)$ for all $e \in E$.  
\end{definition}

We can also pull-back valuations by linear maps. 
\begin{definition}[Pull-back of a valuation] \label{def-pull-back-valuation}
Let $F: E \to E'$ be a linear map between finite dimensional $\k$-vector spaces $E$ and $E'$. For a valuation $w: E' \to \overline{\R}$, we define the pull-back $F^*(w): E \to \overline{\R}$ by:
$$F^*(w) = w \circ F.$$
We note that $F^*(w)$ satisfies the conditions (1) and (2) in the definition of a valuation (Definition \ref{def-val}) except that it may not be true that $F^*(w)(e) = \infty$ only for $e = 0$. Thus, $F^*(w)$ is a valuation on $E$ only if $F$ is one-to-one.     
\end{definition}

\subsection{The space of valuations and Tits building}
\label{subsec-val-Tits}
The space of valuations $\tilde{\B}(E)$ gives a nice way of constructing the geometric realization of the Tits building of the general linear group $\GL(E)$. Below, we briefly review the notion of the Tits building of a linear algebraic group and explain how the space of valuations is related to the Tits building of $\GL(E)$. For more details and other related material we refer the reader to \cite[Sections 1.2 and 1.3]{KM-tpb}.

We begin by recalling the definition of an (abstract) building (see \cite[Definition 4.1]{Abramenko-Brown}). 
The axioms B(2) and B(3) make an appearance later in the section (Lemmas \ref{lem-B3-bldg} and \ref{lem-B3-val}) as well as in Section \ref{subsec-tvb-positivity}.
\begin{definition}[Building]  \label{def-bldg}
A \emph{building} is an (abstract) simplicial complex $\Delta$ that can be expressed as a union of subcomplexes called \emph{apartments} satisfying the following axioms:
\begin{itemize}
\item[(B1)] Each apartment is a Coxeter complex (associated to a Coxeter group).
\item[(B2)] For any two simplices in $\Delta$ there is an apartment containing both of them. 
\item[(B3)] For any two apartments, there is a (simplicial) isomorphism between them that fixes all the simplices in their intersection. 
\end{itemize}
\end{definition}

To any linear algebraic group $G$ over a field $\k$, there corresponds a building, which we denote by $\Delta(G)$, called its \emph{Tits building}. 
The building $\Delta(G)$ and its apartments encode the relative position of parabolic subgroups and maximal tori in $G$. Each apartment in $\Delta(G)$ is a copy of the Coxeter complex associated to the Weyl group of $G$. Tits buildings are among the most important examples of the general notion of a building. 

The simplices in $\Delta(G)$ correspond to the parabolic subgroups in $G$ ordered by reverse inclusion. That is, for parabolic subgroups $P$, $Q$ with simplices $\sigma_P$, $\sigma_Q$ we have $\sigma_P < \sigma_Q$ if $Q \subset P$. 
The apartments in $\Delta(G)$ correspond to maximal tori in $G$.  
For a maximal torus $H$ and a parabolic subgroup $P$, the simplex $\sigma_P$ lies in the apartment $\Sigma(H)$ if $H \subset P$. 

Since all the parabolic subgroups contain the radical $R(G)$ of $G$, one sees that, as simplicial complexes, the Tits building of $G$ and its semisimple quotient $G/R(G)$ coincide. 

The simplicial complex $\Delta(G)$ has a natural \emph{geometric realization}. That is, one constructs a topological space $\B(G)$ together with a triangulation where the simplices in the triangulation of $\B(G)$ correspond to the simplices in $\Delta(G)$. In $\B(G)$, each apartment becomes a triangulation of a sphere (hence $\Delta(G)$ is also referred to as a \emph{spherical building}). We skip the details of the construction of $\B(G)$ (see \cite[Sections 1.2 and 1.3]{KM-tpb}).

\begin{remark}   \label{rem-bldg-notation}
While in our notation we distinguish between the building $\Delta(G)$ (which is a simplicial complex) and its geometric realization $\B(G)$ (which is a topological space), by abuse of terminology we may refer to both $\Delta(G)$ and $\B(G)$ as the Tits building of $G$.  
\end{remark}

The example that interests us in this paper is the Tits building of the general linear group $\GL(E)$ which we now describe. As usual let $E$ be an $r$-dimensional vector space over a field $\k$. Note that, as a simplicial complex, the Tits building of the general linear group $\GL(E)$ and those of $\PGL(E)$ and $\SL(E)$ are all the same. For simplicity, let us denote the Tits building of $\GL(E)$ by $\Delta(E)$. It can be described as follows: the simplices (parabolic subgroups) correspond to the flags of subspaces in $E$:
$$F_\bullet = (\{0\} \subsetneqq F_1 \subsetneqq \cdots \subsetneqq F_k = E).$$
The apartments (maximal tori) correspond to the frames $L = \{L_1, \ldots, L_r\}$, i.e. direct sum decompositions of $E$ into $1$-dimensional subspaces $L_1, \ldots, L_r$. 
A flag $F_\bullet$ is said to be \emph{adapted} to a frame $L$ if every subspace $F_i$ is spanned by a subset of the frame $L$. The apartment corresponding to $L$ consists of all flags adapted to $L$. 
We note that the collection of flags adapted to $L$ can be identified with the collection of ordered partitions of $\{1, \ldots, r\}$. By an ordered partition we mean a $k$-tuple of subsets $(J_1, \ldots, J_k)$, for $0 < k \leq r$, that partition $\{1, \ldots, r\}$. To an ordered partition $J=(J_1, \ldots, J_k)$ there corresponds the flag $F_{L, J, \bullet}$ defined by:
$$F_{L, J, i} = \sum_{j \in J_1 \cup \cdots \cup J_i} L_j, \quad i=1, \ldots, k.$$  
We thus see that each apartment is isomorphic to the Coxeter complex of the symmetric group $S_r$.

We now turn to the task of describing the geometric realization of the Tits building $\Delta(E)$. For simplicity, let us denote this space by $\B(E)$. The space $\tilde{\B}(E)$ of $\R$-valued valuations $E$ gives a natural and convenient way to construct the geometric realization $\B(E)$ as follows. First, we note that if $v$ is a valuation on $E$, then for any $m, c$ with $m > 0$, $mv+c$ is also a valuation on $E$. For valuations $v$, $w$ on $E$, we say that $v \sim w$ if there exist real numbers $m$, $c$ with $m > 0$ such that $w = mv + c$. One has the following:
\begin{proposition}[Tits building as space of valuations]    \label{prop-bldg-val}
The space $\B(E)$ can be realized as the quotient space $$\B(E) = \tilde{\B}(E) / \sim,$$ 
i.e., the space of equivalence classes of $\R$-valued valuations on $E$. Moreover, for each frame $L$, the corresponding apartment $A(L)$ can be realized as the quotient $\tilde{A}(L) / \sim$, i.e., the space of equivalence classes of valuations adapted to $L$. 
\end{proposition}
Thus, we will also refer to $\tilde{\B}(E)$ as the \emph{extended Tits building} or the \emph{cone over the Tits building} of $\GL(E)$.
We refer to $\tilde{A}(L)$, the set of valuations adapted to $L$, as an \emph{extended apartment} or a \emph{cone over an apartment}.


Recall (paragraph before Definition \ref{def-space-of-val}) that, for each frame $L$, the extended apartment $\tilde{A}(L)$ can be regarded as a copy of $\R^r$. It is then easy to see that the apartment $A(L) = \tilde{A}(L) / \sim$ can be regarded as a copy of the unit $(r-2)$-dimensional sphere. To see $\tilde{A}(L) \cong \R^r$, for any $a=(a_1, \ldots, a_r) \in \R^r$, define a valuation $v_{L, a}: E \to \overline{\R}$ by:
\begin{equation}  \label{equ-v-L-a}
v_{L, a}(e) = \min\{a_i \mid e_i \neq 0\},
\end{equation}
where, as before, $e = \sum_i e_i$ with $e_i \in L_i$.
We identify $\tilde{A}(L)$ with $\R^r$ by identifying $v_{L, a}$ with $a$.   

{We end this section by illustrating the axioms B(2) and (B3) in Definition \ref{def-bldg} in the case of the Tits building $\Delta(E)$ of the group $\GL(E)$. Interestingly, these show up in the study of splitting and positivity properties of toric vector bundles (see \cite[Lemma 5.4]{HMP}, Corollary \ref{cor-equiv-split-P1} and Section \ref{subsec-tvb-positivity}). 

The axiom (B2) for the Tits building $\Delta(E)$ becomes the following well-known linear algebra fact (see \cite[Section 9.2]{Garrett}):
\begin{lemma}  \label{lem-bldg-B2}
Given any two flags in $E$, there is a frame such that both flags are adapted to it.
\end{lemma}

Next lemma is the linear algebra fact that implies the axiom (B3) for the Tits building $\Delta(E)$ (see \cite[Section 9.2]{Garrett}). Let $L$, $L'$ be frames for $E$.
Let $\pi: \{1, \ldots, r\} \to \{1, \ldots, r\}$ be a permutation. It gives a bijection between the frames $L$ and $L'$ by $L_i \mapsto L'_{\pi(i)}$ which in turn induces a simplicial isomorphism between the corresponding apartments as follows: for every ordered partition $J = (J_1, \ldots, J_k)$, we send the flag $F_{L, J, \bullet}$ to $F_{L', \pi(J), \bullet}$.
\begin{lemma}   \label{lem-B3-bldg}
With notation as above, there is a permutation $\pi: \{1, \ldots, r\} \to \{1, \ldots, r\}$ such that the corresponding simplicial isomorphism between the apartments of $L$ and $L'$, fixes every flag that is adapted to both $L$ and $L'$. 
\end{lemma}

Similarly, a permutation $\pi: \{1, \ldots, r\} \to \{1, \ldots, r\}$ gives an identification of $\tilde{A}(L)$ with $\tilde{A}(L')$ as follows: for each $a=(a_1, \ldots, a_r)$, we send the valuation $v_{L, a}$ to $v_{L', \pi^{-1}(a)}$. Lemma \ref{lem-B3-bldg} immediately implies the following statement about valuations adapted to $L$ and $L'$.
\begin{lemma}  \label{lem-B3-val}
With notation as above, there is a permutation $\pi: \{1, \ldots, r\} \to \{1, \ldots, r\}$ such that the corresponding identification of $\tilde{A}(L)$ and $\tilde{A}(L')$ fixes every valuation that is adapted to both $L$ and $L'$. 
\end{lemma}
}

\section{Toric vector bundles as piecewise linear maps}
\label{sec-tvbs-plm}
In this section we interpret the Klyachko data of compatible filtrations on $E$ as a piecewise linear map $\Phi$ to the space of valuations $\tilde{\B}(E)$. Moreover, we interpret the criteria in \cite{HMP} and \cite{DJS} for ampleness and global generation of a toric vector bundle $\E$ as convexity conditions on its corresponding piecewise linear map $\Phi = \Phi_\E$. 

We emphasize that the concept of a piecewise linear map $\Phi$ is not quite new and is basically the data of interpolation filtrations introduced in \cite{Payne-cover}. It is also a special case of a more general construction in \cite{KM-tpb} for classifying toric principal bundles.  
\subsection{Toric vector bundles as piecewise linear maps to space of valuations}   \label{subsec-PL-maps}
We recall from Proposition \ref{prop-preval-filtration} that the set of $\Z$-filtrations on $E$ can be identified with the set of integral valuations on $E$. This gives a convenient way to package the Klyachko data (of compatible filtrations) of a toric vector bundle as a piecewise linear map into the space of valuations. 

\begin{definition}[Piecewise linear map to space of valuations]   \label{def-plm}
With notation as before, a map $\Phi: |\Sigma| \to \tilde{\B}(E)$ is a \emph{piecewise linear map} if the following hold: For any $\sigma \in \Sigma$, there is a frame $L$ for $E$ such that $\Phi(\sigma)$ lands in an (extended) apartment $\tilde{A}(L)$. Moreover, we require that the restriction $\Phi_{|\sigma}: \sigma \to \tilde{A}(L)$ to be linear, i.e. it is the restriction of a linear map from $N_\R$ to $\tilde{A}(L)$. We say that a piecewise linear map $\Phi$ is \emph{integral} if $\Phi$ sends lattice points to lattice points, i.e. $\Phi(N \cap |\Sigma|) \subset \tilde{\B}_\Z(E)$.
\end{definition}

The space of piecewise linear maps on $\Sigma$ can be turned into a category:
\begin{definition}[Morphism of piecewise linear maps]   \label{def-morphism-plm}
Let $F:E \to E'$ be a linear map between finite dimensional $\k$-vector spaces $E$ and $E'$. Let $\Phi$ and $\Phi'$ be piecewise linear maps from $|\Sigma|$ to $\tilde{\B}(E')$ and $\tilde{\B}(E')$ respectively. We say that $F$ gives a morphism $\Phi \to \Phi'$ if for any $x \in |\Sigma|$ we have: $$F^*(\Phi'(x)) \geq \Phi(x).$$
Recall that $F^*$ is the pull-back of a valuation by the linear map $F$ (Definition \ref{def-pull-back-valuation}) and $\geq$ is the partial order on the space of functions on $E$ defined as follows: for $w_1$, $w_2: E \to \overline{\R}$, we say $w_1 \geq w_2$ if $w_1(e) \geq w_2(e)$ for all $e \in E$ (cf. Definition \ref{def-po-valuation}). 
\end{definition}

It is straightforward to see that the space of piecewise linear maps on $\Sigma$ together with the above notion of morphism forms a category. Moreover, the integral piecewise linear maps on $\Sigma$ form of a subcategory. One observes that the subcategory of integral piecewise linear maps on $\Sigma$ is the same as the category of compatible $\Z$-filtrations with respect to $\Sigma$.

The Klyachko classification of toric vector bundles (Theorem \ref{th-Klyachko}) can be restated as follows:
\begin{theorem}[Classification of toric vector bundles in terms of piecewise linear maps]   \label{th-Klyachko-plm}
The category of toric vector bundles on $X_\Sigma$ is naturally equivalent to the category of integral piecewise linear maps to $\tilde{\B}(E)$, for all finite dimensional $k$-vector spaces $E$. 
\end{theorem}

\begin{example}[Tangent bundle of $\mathbb{P}^n$]  \label{ex-tangent-bundle-Pn} Let us give the Klaychko filtrations for the tangent bundle of a projective space and then interpret this data as a piecewise linear map. Consider the projective space $\mathbb{P}^n$ and let $\E = T\mathbb{P}^n$ be its tangent bundle. Identify the lattices $M$ and $N$ with $\Z^n$ and let $\Sigma \in \R^n$ be the fan of $\mathbb{P}^n$. The primitive vectors on the rays in $\Sigma$ are the vectors 
$\{\v_1, \ldots, \v_{n+1}\}$ where $\v_i$ is the $i$-th standard basis vector and $\v_{n+1} = - \v_1 - \cdots - \v_n$. We write 
$\rho_i$ for the ray generated by $\v_i$. One identifies the fibre $E$ of $T\mathbb{P}^n$ over the identity of the torus $T$ with $N \otimes_\Z \k = \k^n$. Hence, the vectors $\v_1, \v_2,\ldots, \v_n$ also form the standard basis for $E = \k^n$. One computes that the Klyachko filtrations are given as follows:
$$E_j^{\rho_i} =
\begin{cases}
E & j \leq 0 \cr 
\text{span}(\v_i) & j=1 \cr
0 & j > 0
\end{cases}
$$
The $n$-dimensional cones in the fan are $\{\sigma_1, \ldots, \sigma_n\}$ where $\sigma_i$ is the cone spanned by $\{\v_1, \ldots, \hat{\v}_i, \ldots, \v_{n+1}\}$, $1 \leq i \leq n+1$ (here $\hat{\v_i}$ means this vector is removed). For each $\sigma_i$ we have a basis of $E$ and a multiset of characters $u(\sigma_i)$ appearing in Klyachko's compatibility condition. Let $\{\w_1, \ldots, \w_n\}$ denote the standard basis elements in $M \cong \Z^n$. One computes that for any $\sigma_i$ the corresponding basis is $\{\v_1, \ldots, \hat{\v}_i, \ldots, \v_{n+1}\}$. Moreover, for $1 \leq i \leq n$, $u_{\sigma_i}
= \{ \w_1-\w_i, \w_2 - \w_i, \ldots, \w_{i-1} - \w_i, -\w_i, \w_{i+1} - \w_i, \ldots, \w_n - \w_i \}$ and $u_{\sigma_{n+1}} = \{\w_1, \ldots, \w_n\}$. 
The peicewise linear map $\Phi = \Phi_{T\mathbb{P}^n}: |\Sigma| \to \B(\GL(E))$ is then given as follows. Let $x \in N_\Q$ lie in the maximal cone $\sigma_i$. Let $0 \neq e \in E$ and let us write it in the basis corresponding to $\sigma_i$ as $e = \sum_j e_j$. Then: 
$$\Phi(x)(e) = \min\{ \langle x, u_{\sigma_i, j} \rangle \mid e_j \neq 0 \}.$$
\end{example}

To illustrate the usefulness of packaging the Klyachko data as a piecewise linear map into an (extended) Tits building, we state a result on equivariant splitting of toric vector bundles. This is a restatement of a result of Klyachko (\cite[Corollary 2.2.3]{Klyachko}) in our language. A toric vector bundle is said to \emph{split equivariantly} if it is equivariantly isomorphic to a direct sum of toric line bundles. 

\begin{proposition}[Criterion for equivariant splitting]  \label{prop-equiv-split}
Let $\E$ be a toric vector bundle with the corresponding piecewise linear map $\Phi_\E: |\Sigma| \to \tilde{\B}(E)$. Then $\E$ splits equivariantly if and only if the image of $\Phi_\E$ lands in a single (extended) apartment $\tilde{A}(L)$. 
\end{proposition}

The following is an immediate corollary of Proposition \ref{prop-equiv-split} and the axioms of a building (Definition \ref{def-bldg}(B2)). It can be found in \cite[Section 2.3, Example 3]{Klyachko}, as well as \cite[Corollary 5.5]{HMP}. 
\begin{corollary}[Equivariant splitting on $\mathbb{P}^1$] \label{cor-equiv-split-P1}
Any toric vector bundle on $\mathbb{P}^1$ is equivariantly split. 
\end{corollary}
\begin{proof}
The fan $\Sigma$ of $\mathbb{P}^1$ consists of two rays $\R_{\geq 0}$ and $\R_{\leq 0}$ with primitive vectors $1$ and $-1$ respectively. 
Let $\E$ be a toric vector bundle on $\mathbb{P}^1$ and let $\Phi: |\Sigma| = \R \to \tilde{\B}(E)$ be the piecewise linear map corresponding to $\E$. 
By the axioms B(2) in the definition of a building (Definition \ref{def-bldg}) there is an (extended) apartment $\tilde{A}(L)$ that contains both $\Phi(1)$ and $\Phi(-1)$. The corollary now follows from Proposition \ref{prop-equiv-split}.
\end{proof}

Finally, the following is a slightly refined version of Corollary \ref{cor-equiv-split-P1}. It is from \cite[Corollary 5.10]{HMP}. Let $\sigma, \sigma' \in \Sigma$ be full dimensional cones in a fan $\Sigma$ such that their intersection $\tau = \sigma \cap \sigma'$ is a comdimension $1$ face. Let $C$ be a $T$-invariant curve in $X_\Sigma$ corresponding to the cone $\tau$. Let $u, u' \in M$ be such that $u-u'$ is orthogonal to $\tau$. Let $\L_u$, $\L_{u'}$ be trivial line bundles on affine toric charts $U_\sigma$, $U_{\sigma'}$ and equipped with $T$-linearizations via characters $\chi^u$, $\chi^{u'}$ respectively. Then one can construct a line bundle $\L_{u, u'}$ on $U_\sigma \cup U_{\sigma'}$ by gluing the line bundles $\L_u$ and $\L_{u'}$ via the transition function $\chi^{u-u'}$ which is regular and invertible on $U_\tau$. Let $\v_\tau \in \sigma$ be the vector that is dual to the primitive generators of $\tau^\perp$. One shows that the line bundle ${\L_{u,u'}}_{|C}$ is isomorphic to the line bundle $\mathcal{O}_{\mathbb{P}^1}(\langle u, \v_\tau \rangle D_1 - \langle  u', \v_\tau \rangle D_2)$ on $\mathbb{P}^1$, where $D_1, D_2$ are the $T$-fixed points in $\mathbb{P}^1$ (under the isomorphism $C \cong \mathbb{P}^1$, $x_\sigma \mapsto D_1$ and $x_{\sigma'} \mapsto D_2$). 

\begin{corollary} \label{cor-equiv-split-curve}
Let $\E$ be a toric vector bundle over $X_\Sigma$. Let $u(\sigma) = \{u_1, \ldots, u_r\}$ and $u(\sigma') = \{u'_1, \ldots, u'_r\}$ be the multisets of characters corresponding to $\E_{|U_\sigma}$ and $\E_{|U_\sigma'}$ respectively. Then there exists a permutation $\pi: \{1, \ldots, r\} \to \{1, \ldots, r\}$ such that the $T$-equivariant vector bundle $\E_{|C}$ on the $T$-equivariant curve $C \cong \mathbb{P}^1$ splits $T$-equivariantly as a sum of line bundles $\L_{u_1, u'_{\pi(1)}} \oplus \cdots \oplus \L_{u_r, u'_{\pi(r)}}$. Moreover, the collection of pairs $\{(u_1, u'_{\pi(1)}), \ldots, (u_r, u'_{\pi(r)}) \}$ is uniquely determined. 
\end{corollary}

\subsection{Positivity criteria in terms of convexity of piecewise linear maps}  
\label{subsec-tvb-positivity}
In this section we look at the positivity notions for a toric vector bundle $\E$ (\cite{HMP} and \cite{DJS}) and interpret them as convexity properties on the corresponding piecewise linear map $\Phi_\E$. More precisely, we define two notions of convexity for piecewise linear maps to $\tilde{\B}(E)$ and show that, extending the case of toric line bundles, these notions correspond to positivity properties (ampleness and global generation) of associated toric vector bundles. 


First let $f: N_\R \to \R$ be a piecewise linear function with respect to a complete fan $\Sigma$. In this case, the convexity of $f$ can be characterized as follows. Let $\sigma, \sigma' \in \Sigma(n)$ be two full dimensional cones that intersect in an $(n-1)$-dimensional cone $\tau$. Let $f_\sigma: N_\R \to \R$ (respectively $f_{\sigma'}: N_\R \to \R$) be the linear function that coincides with $f_{|\sigma}: \sigma \to \R$ (respectively $f_{|\sigma'}: \sigma' \to \R$). Then $f$ is {\it convex} on $\sigma \cup \sigma'$ if for any $x \in \sigma$ we have $f_{\sigma'}(x) \leq f(x)$ and for any $x' \in \sigma'$, $f_\sigma(x') \leq f(x')$. And $f$ is {\it convex} if it is convex on any $\sigma \cup \sigma'$ (see Figure \ref{fig-convex1}). 

We can generalize this description of a convex piecewise linear function to piecewise linear maps to $\tilde{\B}(E)$ as follows. First, we recall that, for a frame $L = \{L_1, \ldots, L_r\}$, $\tilde{A}(L)$ denotes the set of valuations adapted to $L$ (Definition \ref{def-space-of-val}). As before, let $\sigma$, $\sigma' \in \Sigma(n)$ be maximal cones with $\tau = \sigma \cap \sigma' \in \Sigma(n-1)$. We would like to say when $\Phi: |\Sigma| \to \tilde{\B}(E)$ is convex on $\sigma \cup \sigma'$. 
By definition there are frames $L=\{L_1, \ldots, L_r\}$, $L'=\{L'_1, \ldots, L'_r\}$ such that $\Phi_{|\sigma}: \sigma \to \tilde{A}(L)$ and $\Phi_{|\sigma'}: \sigma' \to \tilde{A}(L')$ are given by linear maps. In other words, there are $\{u_1, \ldots, u_r\}$ and $\{u'_1, \ldots, u'_r\}$ such that for any $x \in \sigma$, $x' \in \sigma'$ and $0 \neq e \in E$ we have:
$$\Phi_{|\sigma}(x)(e) = \min\{ \langle x, u_i \rangle \mid e_i \neq 0 \},$$
$$\Phi_{|\sigma'}(x')(e) = \min\{ \langle x', u'_j \rangle \mid e'_j \neq 0 \}.$$
Here  $e = \sum_i e_i$, $e_i \in L_i$, and $e = \sum_j e'_j$, $e'_j \in L'_j$, are decompositions of $e$ according to the frames $L$ and $L'$ respectively. In particular, for $x \in \tau$ and $0 \neq e \in E$ we have:
$$\Phi(x)(e) = \min\{ \langle u_i, x \rangle \mid e_i \neq 0 \} = \min\{ \langle u'_j, x \rangle \mid e'_j \neq 0 \}.$$
Let $\pi: \{1, \ldots, r\} \to \{1, \ldots, r\}$ be the permutaiton in Corollary \ref{cor-equiv-split-curve}. In particular, for any $x \in \tau$ we have $\langle u_i, x \rangle = \langle u'_{\pi(i)}, x \rangle$, for all $i=1, \ldots, r$, i.e. $u_i - u'_{\pi(i)} \in \tau^{\perp}$. We use the permutation $\pi$ to define two linear maps $T: \sigma' \to \tilde{A}(L')$, $T': \sigma \to \tilde{A}(L)$. The linear map $T$ (respectively $T'$) should be thought of as an extension of $\Phi_{|\sigma}$ (respectively $\Phi_{|\sigma'}$) to $\sigma'$ (respectively $\sigma$). For any $x \in \sigma$, $x' \in \sigma'$ and $0 \neq e \in E$ put:
\begin{equation}   \label{equ-T-T'} 
T(x')(e) = \min\{ \langle u_i, x' \rangle \mid e'_{\pi(i)} \neq 0 \}, \quad
T'(x)(e) = \min\{ \langle u'_{\pi(i)}, x \rangle \mid e_i \neq 0 \}.
\end{equation}

\begin{figure}
\includegraphics[height=5cm]{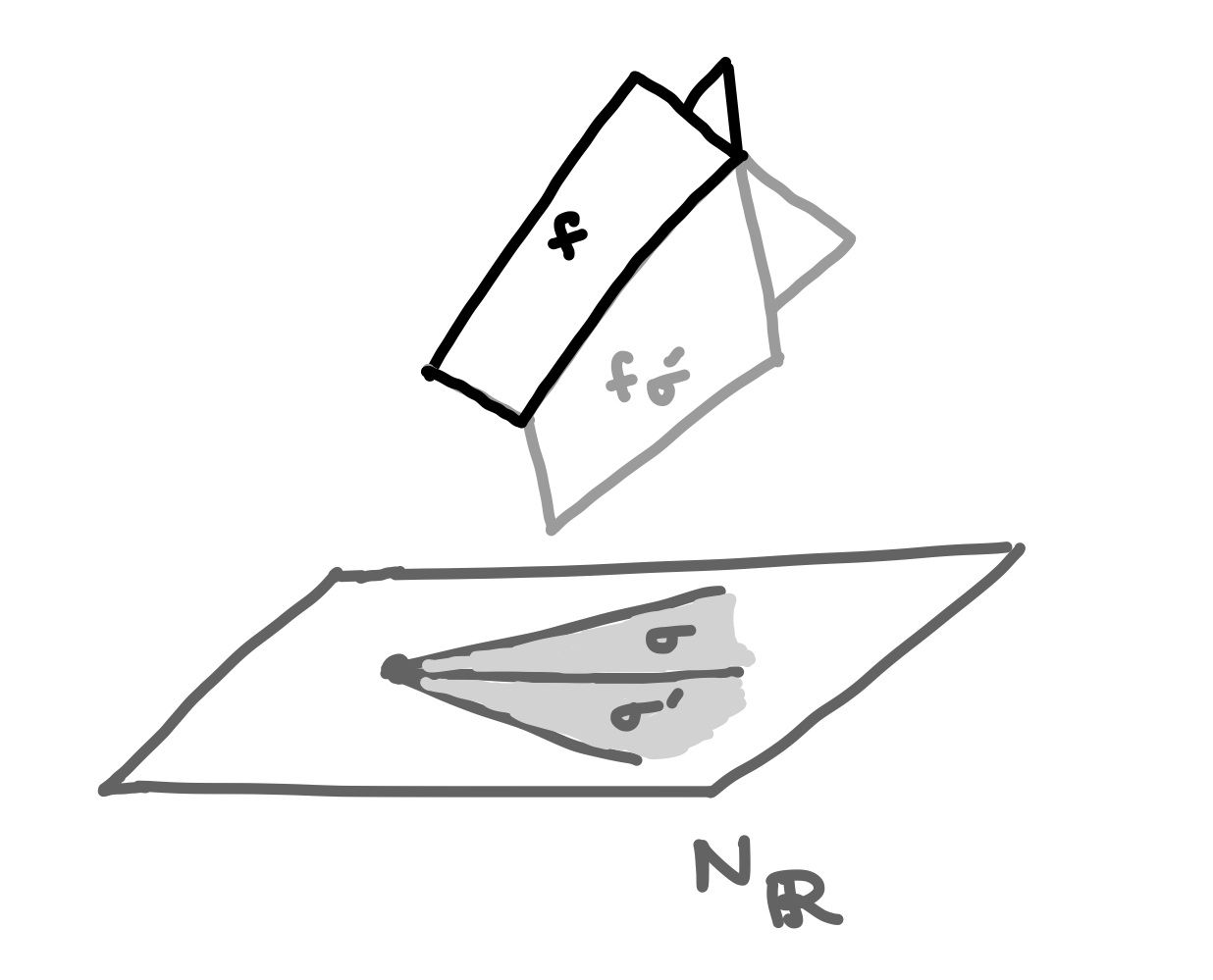}
\caption{Visualization of criterion for convexity of a piecewise linear function} 
\label{fig-convex1}
\end{figure}

\begin{definition}[(Buildingwise) convex map] \label{def-buildingwise-convex}
We say that a piecewise linear map $\Phi: |\Sigma| \to \tilde{\B}(E)$ is {\it (buildingwise) convex} if the following holds for any $\sigma, \sigma' \in \Sigma(n)$ with $\tau=\sigma \cap \sigma' \in \Sigma(n-1)$:
\begin{align}  \label{equ-T-T'-convexity-condition}
T(x') &\leq \Phi(x'), \quad \forall x' \in \sigma', \cr
T'(x) &\leq \Phi(x), \quad \forall x \in \sigma,
\end{align}
where $\leq$ is the partial order on the space of valuations $\tilde{\B}(E)$ as in Definition \ref{def-po-valuation}. 
\end{definition}

\begin{definition}[(Buildingwise) strictly convex map]   \label{def-strictly-convex}
With notation as above, $\Phi: |\Sigma| \to \tilde{\B}(E)$ is a (buildingwise) strictly convex map if the following holds for any $\sigma, \sigma' \in \Sigma(n)$ with $\tau=\sigma \cap \sigma' \in \Sigma(n-1)$:
\begin{align}  \label{equ-T-T'-st-convexity-condition}
T(x') &< \Phi(x'), \quad \forall x' \in \sigma' \setminus \tau,  \cr
T'(x) &< \Phi(x), \quad \forall x \in \sigma \setminus \tau. 
\end{align}
\end{definition}

One verifies that the above definitions are indeed independent of the choice of the bijection $\pi: L \to L'$ in Lemma \ref{lem-B3-val}.
Below we see that the buildingwise convexity (respectively strict buildingwise convexity) of a piecewise linear map is equivalent to the corresponding toric vector bundle being nef (respectively ample). 

\begin{remark}   \label{rem-buildingwise-convex-upper-graph-convex}
We expect that buildingwise convexity of a piecewise linear map $\Phi$ is equivalent to the upper graph of $\Phi$ being a convex subset of $N_\R \times \tilde{\B}(E)$ in a suitable sense. Here we define the upper graph of $\Phi$ as 
$\{(x, v) \mid v \text{ and } \Phi(x) \text{ lie in the same apartment and } v \geq \Phi(x)\} \subset N_\R \times \tilde{\B}(E)$, where as above, $\geq$ denotes the partial order on the space of valuations. 
\end{remark}

There is an alternative way to define convexity of a real-valued piecewise linear function. Let $f: N_\R \to \R$ be a piecewise linear function with respect to a complete fan $\Sigma$ in $N_\R$. For each maximal cone $\sigma \in \Sigma(n)$ let $f_\sigma: N_\R \to \R$ be the linear function that coincides with $f$ on $\sigma$. Then $f$ is {\it convex} if for any maximal cone $\sigma \in \Sigma$ the graph of $f$ lies above that of of $f_\sigma$, that is, for any $x \in N_\R$ we have $f_\sigma(x) \leq f(x)$ (see Figure \ref{fig-convex2}).

Generalizing the above, we define another version of convexity of a piecewise linear map into $\tilde{\B}(E)$. In general, this version of convexity turns out to be different from the previous one (Definition \ref{def-buildingwise-convex}). We see below that this notion of convexity is equivalent to the corresponding toric vector bundle being globally generated. 

With notation as above, let $\Phi: |\Sigma| \to \tilde{\B}(E)$ be a piecewise linear map. For every maximal cone $\sigma \in \Sigma$, let $L_\sigma = \{L_{\sigma, 1}, \ldots, L_{\sigma, r}\}$ and $u_\sigma = \{u_{\sigma, 1}, \ldots, u_{\sigma, r}\} \subset M_\R$ be the corresponding frame and multiset defining the linear map $\Phi_{|\sigma}$. Let $S_\sigma: N_\R \to \tilde{A}(L_\sigma)$ be the linear map that coincides with $\Phi$ on $\sigma$, that is, for every $x \in N_\R$: 
\begin{equation*}
S_\sigma(x)(e) = \min\{ \langle x, u_{\sigma, i} \rangle \mid e_i \neq 0\},
\end{equation*} 
for any $0 \neq e \in E$ with $e = \sum_i e_i$ its decomposition with respect to the frame $L_\sigma$.

\begin{figure}
\includegraphics[height=5cm]{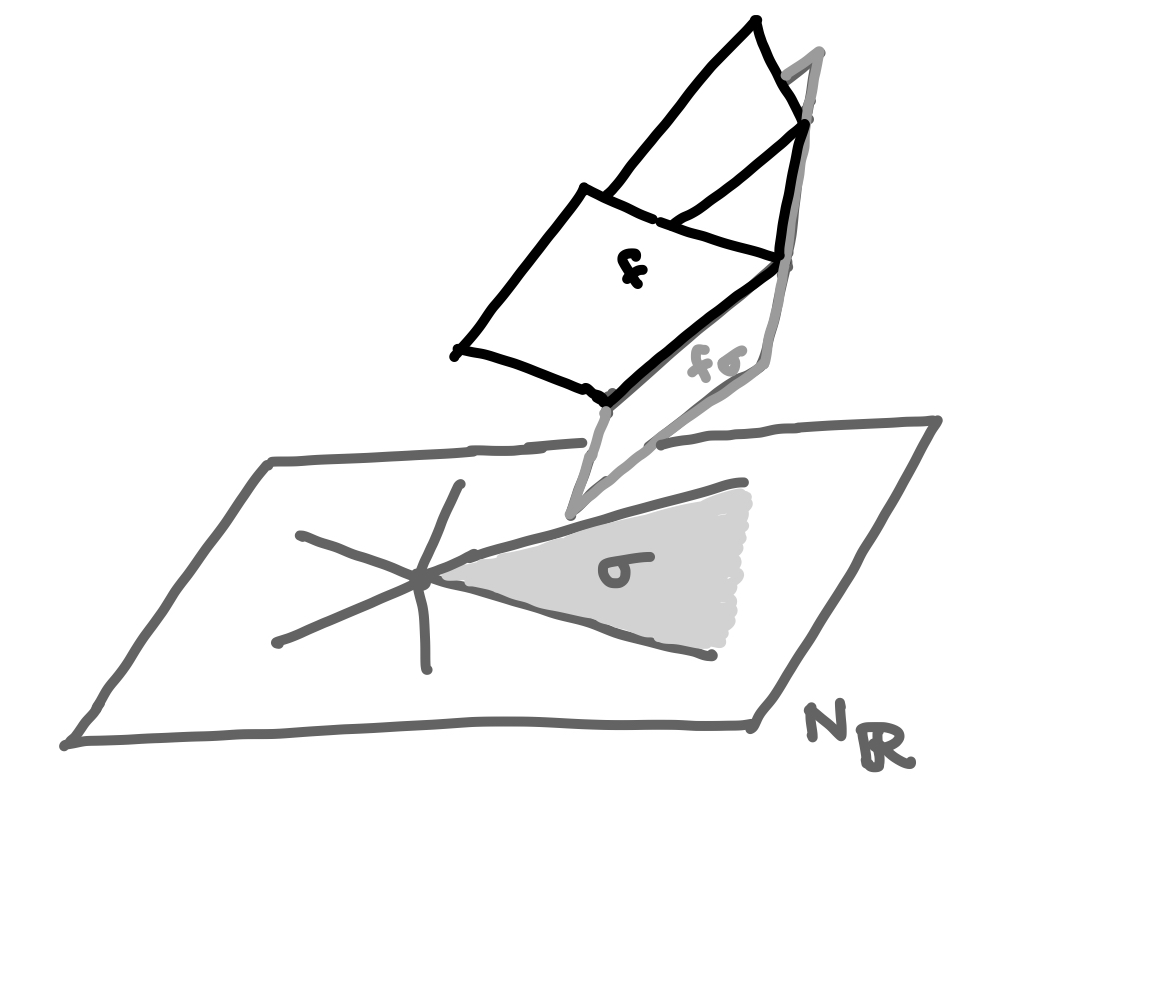}
\caption{Visualization of the second criterion for convexity of a piecewise linear function}
\label{fig-convex2}
\end{figure}

\begin{definition}[(Fanwise) convex map]    \label{def-fanwise-convex}
We say that $\Phi: |\Sigma| \to \tilde{\B}(E)$ is {\it fanwise convex} if the following holds. For every maximal cone $\sigma \in \Sigma(n)$, there exists a compatible frame $L_\sigma$, such that the graph of $\Phi$ lies under the graph of the linear map $S_\sigma$. That is:\begin{equation}  \label{equ-fanwise-convex}
S_\sigma(x) \leq \Phi(x), \quad \forall x \in N_\R,
\end{equation}
where as before, $\geq$ is the partial order on the space of valuations $\tilde{\B}(E)$.
\end{definition}


Note that in \eqref{equ-T-T'-convexity-condition}  we require the inequality to hold for $x \in \sigma$ (or $x \in \sigma'$) while in \eqref{equ-fanwise-convex} we want the similar inequality to hold for all $x \in N_\R$. 


Finally, we relate the notions of ample, nef and globally generated for a toric vector bundle with the notions of convexity of piecewise linear maps discussed above. These generalize the familiar statements for toric line bundles and $\R$-valued piecewise linear functions (\cite[Chapter 6]{CLS}). 

The following gives criteria for nef and ampleness of a toric vector bundle in terms of the (buildingwise) convexity of the corresponding piecewise linear map. It is a corollary of \cite[Theorem 2.1]{HMP}. 
\begin{theorem}   \label{th-ample-strictly-convex}
A toric vector bundle $\mathcal{E}$ over a complete toric variety is nef (respectively ample) if and only if the corresponding piecewise linear map $\Phi_{\mathcal{E}}$ is buildingwise convex (respectively strictly buildingwise convex), in the sense of Definition \ref{def-buildingwise-convex} (respectively Definition \ref{def-strictly-convex}).
\end{theorem}
\begin{proof}
By \cite[Theorem 2.1]{HMP}, a toric vector bundle is nef (respectively ample) if and only if its restriction to any $T$-invariant curve is nef (respectively ample). More precisely, let $C$ be a $T$-invariant curve in $X_\Sigma$ corresponding to a cone $\tau \in \Sigma(n-1)$. Since $\Sigma$ is complete there are maximal cones $\sigma, \sigma' \in \Sigma(n)$ with $\tau = \sigma \cap \sigma'$. Let $u, u' \in M$ be such that $u-u'$ is orthogonal to $\tau$. Let $\L_{u,u'}$ be the corresponding line bundle on $U_\sigma \cup U_{\sigma'}$ as in the paragraph before Corollary \ref{cor-equiv-split-curve}. We recall that if $\L_u$, $\L_{u'}$ are trivial line bundles on affine toric charts $U_\sigma$, $U_{\sigma'}$ with $T$-linearizations given by characters $\chi^u$, $\chi^{u'}$, then $\L_{u, u'}$ is the line bundle on $U_\sigma \cup U_{\sigma'}$ constructed by gluing $\L_u$ and $\L_{u'}$ via the transition function $\chi^{u-u'}$ which is regular and invertible on $U_\tau$. Let $\v_\tau \in \sigma$ be the vector that is dual to the primitive generators of $\tau^\perp$. One shows that the line bundle ${\L_{u,u'}}_{|C}$ is isomorphic to the line bundle 
$\mathcal{O}_{\mathbb{P}^1}(\langle u, \v_\tau \rangle D_1 - \langle  u', \v_\tau \rangle D_2)$ on $\mathbb{P}^1$, where $D_1, D_2$ are the $T$-fixed points in $\mathbb{P}^1$ (under the isomorphism $C \cong \mathbb{P}^1$, $x_\sigma \mapsto D_1$ and $x_{\sigma'} \mapsto D_2$). The line bundle $\L_{u, u'}$ on $C \cong \mathbb{P}^1$ is nef (respectively ample) if and only $a = \langle u, \v_\tau \rangle - \langle u', \v_\tau \rangle \geq 0$ (respectively $a > 0$). Now let $\pi$ be the permutation used in the definition of $T$ and $T'$. Then the vector bundle $\E_{|C}$ splits equivariantly as a direct sum of line bundles $\L_{u_1, u'_{\pi(1)}|C} \oplus \cdots \oplus \L_{u_r, u'_{\pi(r)}|C}$. For each $i$, let us write $\L_{u_i, u'_{\pi(i)}} = \mathcal{O}_{\mathbb{P}^1}(a_i)$. From the above, we know that $u_i - u'_{\pi(i)}$ is $a_i$ times the primitive generator of $\tau^\perp$ which is positive on $\sigma$. Now $\E_{|C}$ is nef (respectively ample) if and only if $a_i \geq 0$ (respectively $a_i > 0$) for all $i$. On the other hand, with notation as in Definition \ref{def-buildingwise-convex}, the condition $T'(x) \leq \Phi(x)$ means that $$\min\{\langle u'_{\pi(i)}, x \rangle \mid e_i \neq 0 \} \leq \min\{ \langle u_i, x \rangle \mid e_i \neq 0 \}, ~\forall x \in \sigma',~\forall 0 \neq e \in E.$$
For each $i$, taking $e=e_i$ implies that $\langle u'_{\pi(i)}, x \rangle \leq \langle u_i, x \rangle$. This shows that the nefness condition above is equivalent to $\langle u_i - u'_{\pi(i)}, x \rangle \geq 0$, for all $x \in \sigma'$ and all $i$. This in turn is equivalent to  $\langle u_i - u'_{\pi(i)}, \v_\rho \rangle \geq 0$, for all $\rho \in \sigma'(1)$ and all $i$. This finishes the proof.
\end{proof}

The next theorem gives a criterion for global generation of a toric vector bundle in terms of the (fanwise) convexity of the corresponding piecewise linear map. It is a corollary of \cite[Theorem 1.2]{DJS}. 
\begin{theorem}   \label{th-globally-generated-convex}
A toric vector bundle $\mathcal{E}$ over a complete toric variety is globally generated if and only if the corresponding piecewise linear map $\Phi$ is fanwise convex (in the sense of Definition \ref{def-fanwise-convex}).
\end{theorem}
\begin{proof}
For $\sigma \in \Sigma(n)$ let $L_\sigma = \{L_1, \ldots, L_r\}$ be a frame associated to it and corresponding multiset of vectors $\{u_{\sigma, 1}, \ldots, u_{\sigma, r}\}$. For every $L_i \in L_\sigma$ choose $0 \neq e_i \in L_i$. \cite[Theorem 1.2]{DJS} gives a necessary and sufficient condition for $\E$ to be globally generated. In our language of piecewise linear maps, this condition can be stated as follows: For every $\sigma \in \Sigma(n)$, there exists a compatible frame $L_\sigma$ such that for all $1 \leq i \leq r$ the point $u_{\sigma, i}$ lies in the polytope $P_{e_i} = \{y \mid \langle y, \v_\rho  \rangle \leq \Phi(\v_\rho)(e_i),~\forall \rho \in \Sigma(1)\}$. This then implies that $u_{\sigma, i}$ is a vertex of $P_{e_i}$. We would like to show that this condition is equivalent to fanwise convexity of $\Phi$. We note that $u_{\sigma, i}$ lies in $P_{e_i}$ if and only if: 
\begin{equation*} 
\langle u_{\sigma, i},  \v_\rho \rangle \geq \Phi(\v_\rho)(e_i), \quad \forall \rho \in \Sigma(1).
\end{equation*}
Because of piecewise linearity of $\Phi$, this in turn is equivalent to:
\begin{equation*} 
\langle u_{\sigma, i}, x \rangle \geq \Phi(x)(e_i), \quad \forall x \in N_\R
\end{equation*}
The above equation means that the graph of the piecewise linear function $x \mapsto \Phi(x)(e_i)$ lies below that of the linear function $x \mapsto \langle x, u_{\sigma, i}\rangle$, for all $i$, which is equivalent to the definition of fanwise convexity of $\Phi$. 
\end{proof}

\section{Toric vector bundles as valuations}
\label{sec-tvbs-PL-val}
In this section we introduce the notion of a vector space valuation with values in a semilattice $(\Gamma, \wedge)$. We consider the semilattice $(\PL(N, \Z), \min)$ where as usual $\PL(N, \Z)$ is the set of $\Z$-valued piecewise linear functions on a lattice $N \cong \Z^r$. We show that the valuations on $E$ with values in this semilattice classify toric vector bundles $\E$ with fiber $\E_{x_0} = E$  and up to toric pull-backs (Theorem \ref{th-semilattice-preval-toric-vb}). We caution that in this section (unfortunately) we have two different usages of the term \emph{lattice}: first we use lattice to mean a finite rank free abelian group (as in lattices $M$ and $N$), and second by a lattice we mean a meet-join lattice, a kind of partially ordered set.
  
\subsection{Valuations with values in a semilattice}  \label{subsec-semilattice-val}
Let $(\Gamma, \geq, \wedge)$ be a meet-semilattice. That is, $(\Gamma, \geq)$ is a partially ordered set (poset) together with a binary operation $\wedge$ (meet) of greatest lower bound. That is, for any $\gamma, \eta \in \Gamma$, their meet $\gamma \wedge \eta$ is $\leq$ both $\gamma$ and $\eta$, and whenever we have $\mu \leq \gamma$, $\mu \leq \eta$, for some $\mu \in \Gamma$, then $\mu \leq \gamma \wedge \eta$. We also assume that $\Gamma$ has a (unique) maximum element denoted by $\infty$.

\begin{definition}[Semilattice valuation]  \label{def-semilattice-val}
Let $\pi: E \to \Gamma$ be a map that satisfies the following:
\begin{itemize}
\item[(a)] For any $e \in E$ and any $0 \neq c \in \k$ we have $\pi(c e) = \pi(e)$.
\item[(b)](Non-Archimedean property) For any $e_1, e_2 \in E$ we have:
\begin{equation}  
\pi(e_1+e_2) \geq \pi(e_1) \wedge \pi(e_2).
\end{equation}
\item[(c)] $\pi(e) = \infty$ if and only if $e = 0$.
\end{itemize}
We call such a map $\pi$ a {\it semilattice valuation} (or just a {\it valuation} for short) on $E$ with values in $\Gamma$. If the semilattice generated by the image of $\pi$ is a finite set we call $\pi$ a \emph{finite valuation}.
\end{definition}
From the definition it follows that for every $\gamma \in \Gamma$ the set:
$$E_{\geq \gamma} = \{e \in E \mid \pi(e) \geq \gamma \}$$
is a vector subspace of $E$.
To any semilattice valuation $\pi$ on $E$ and a subset $S \subset \Gamma$ we associate the arrangement of linear subspaces in $E$: $$\mathcal{A}_{\pi, S} = \{E_{\geq \gamma} \mid \gamma \in S \}.$$  

\begin{lemma}   \label{lem-semilattice-preval-arrangement}
Let $\pi: E \to \Gamma$ be a finite semilattice valuation and let $S \subset \Gamma$ be a subset. 
(1) The arrangement $\mathcal{A}_{\pi, S}$ consists of a finite number of subspaces.  
(2) If the semilattice $\Gamma$ is a lattice, i.e. also has an operation $\vee$ (join) of least upper bound, and $S$ is closed under $\vee$ then the subspace arrangement $\mathcal{A}_{\pi, S} = \{ E_{\geq \gamma} \mid \gamma \in S\}$ is closed under intersection.
\end{lemma}
\begin{proof}
(1) It suffices to show that for any $\gamma \in \Gamma$ there exits $\gamma'$ in the semilattice generated by the image of $\pi$ such that $E_{\geq \gamma} = E_{\geq \gamma'}$. 
Let $\gamma'$ be the greatest lower bound of all elements in the image of $\pi$ which are $\geq \gamma$. This exists since the image of $\pi$ is finite. Then $\gamma' \geq \gamma$ implies that $E_{\geq \gamma'} \subset E_{\geq \gamma}$. Suppose for some $0 \neq e \in E$ we have $\pi(e) \geq \gamma$. From the above we see that $\pi(e) \geq \gamma'$ and hence $E_{\geq \gamma'} = E_{\geq \gamma}$. 
(2) Let $\gamma_1, \ldots, \gamma_s \in S$ and consider $E_{\geq \gamma_1} \cap \cdots \cap E_{\geq \gamma_s}$. We claim that $E_{\geq \gamma_1} \cap \cdots \cap E_{\geq \gamma_s} = E_{\geq \gamma}$ where $\gamma = \gamma_1 \vee \cdots \vee \gamma_s$. If for some $e$ we have $\pi(e) \geq \gamma_i$, for all $i$, then $\pi(e) \geq \gamma_1 \vee \cdots \vee \gamma_s$ and hence $e \in E_{\geq \gamma}$. Conversely, $\pi(e) \geq \gamma_1 \vee \cdots \vee \gamma_s$ implies that $\pi(e) \geq \gamma_i$ for all $i$ and hence $e \in E_{\geq \gamma_1} \cap \cdots \cap E_{\geq \gamma_s}$.
\end{proof}


We end the section by recalling the \emph{matroid} associated to a linear subspace arrangement. Suppose $\mathcal{A} = \{V_1, \ldots, V_s\}$ is an arrangement of linear subspaces in $E$ that is closed under intersection. To $\mathcal{A}$ one naturally associates a matroid $M(\mathcal{A})$ as follows (see \cite[Section 4]{Ziegler}). Note that we use the dual of the definition and statement in \cite[Definition 4.8]{Ziegler}. 
Let $U_1, \ldots, U_m$ be subspaces in $\mathcal{A}$ that are not sum of other subspaces in $\mathcal{A}$. For each $i=1, \ldots, m$, pick a basis $B_i$ for $U_i$ and let $\mathcal{B} = \bigcup_{i=1}^m B_i$. One says that such a spanning set $\mathcal{B}$ is {\it generic} (for the subspace arrangement $\mathcal{A}$) if the following is satisfied: 
For any $B_0 \subset \mathcal{B}$ and $e_i \in B_i \setminus B_0$, $e_i$ lies in the span of $B_0$ if only if the whole $U_i = \Span(B_i)$ lies in the span of $B_0$. The following is known (see \cite[Theorem 4.9]{Ziegler}):
\begin{theorem}[Matroid associated to a subspace arrangement]   \label{th-matroid-of-arrangement}
Let $\mathcal{B}$ be generic with respect to an arrangement of linear subspaces $\mathcal{A}$ that is closed under intersection. Then the matroid structure of the set of vectors $\mathcal{B}$ only depends on $\mathcal{A}$ (i.e. is independent of the choice of $\mathcal{B}$).
\end{theorem} 

The above motivates us to make the following definition which is used in Section \ref{subsec-parliament} in connection to the notion of a parliament of polytopes.
\begin{definition}[Matroids associated to a semilattice valuation]   \label{def-matroid-semilattice-val}
Let $(\Gamma, \geq, \wedge, \vee)$ be a lattice. Let $\pi: E \to \Gamma$ be a finite semilattice valuation and $S \subset \Gamma$ a subset closed under the join operation $\vee$. The \emph{matroid $M(\pi, S)$ associated to $(\pi, S)$} is the (representable) matroid corresponding to the subspace arrangement $\mathcal{A}_{\pi, S}$ (as in Theorem \ref{th-matroid-of-arrangement}). We note that since $S$ is closed under $\vee$, by Lemma \ref{lem-semilattice-preval-arrangement}, the arrangement $\mathcal{A}_{\pi, S}$ is closed under intersection.
\end{definition}

\subsection{Valuations with values in piecewise linear functions and polytopes}  \label{subsec-preval-plf} 
Recall that a function $h: N_\R \to \R$ is piecewise linear if there exists a complete fan $\Sigma$ in $N_\R$ such that $h$ is linear restricted to each cone $\sigma \in \Sigma$. We denote the set of all piecewise linear functions on $N_\R$ by $\PL(N_\R, \R)$. Moreover, we add
a unique ``infinity element'' $\infty$ to $\PL(N_\R, \R)$ which is greater than any other element. We
regard it as the function which assigns value infinity to all points in $N_\R$.
There is a natural partial order on $\PL(N_\R, \R)$ where $\phi_1 \leq \phi_2$ if $\phi_1(x) \leq \phi_2(x)$, $\forall x \in N_\R$. If $\phi_1, \phi_2 \in \PL(N_\R, \R)$ then $\min(\phi_1, \phi_2)$ and $\max(\phi_1, \phi_2)$ also belong to $\PL(N_\R, \R)$. The partial order $\leq$ together with the operations $\min$ and $\max$ give $\PL(N_\R, \R)$ the structure of a lattice. We also denote the set of piecewise linear functions that attain integer values on $N$ by $\PL(N, \Z)$. 
Finally, for a complete fan $\Sigma$, we denote by $\PL(\Sigma, \R)$ the set of piecewise linear functions that are linear on cones in $\Sigma$ and $\PL(\Sigma, \Z)$ the subset of piecewise linear functions that have integer values on $N$. 


\begin{remark}   \label{rem-PL-semialgebra}
Since $\PL(N_\R, \R)$ is closed under addition, the set $\PL(N_\R, \R)$ in fact has structure of a semifield. But in this section we do not address its semifield structure. This semifield is used in Section \ref{sec-tvbs-trop-points} to give a characterization of toric vector bundles as tropical points of linear ideals over this semifield. This idea is explored and expanded in the companion paper \cite{KM-PL} where toric flat families are classified by algebra valuations with values in the semifield $\PL(N, \Z)$.
\end{remark}

In this section we look at finite valuations $\fv$ with values in the semilattice of piecewise linear functions $(\PL(N_\R, \R), \geq, \min)$. 

\begin{definition}[Piecewise linear valuation]  \label{def-PL-val}
Let $\fv: E \to \PL(N_\R, \R)$ be a finite semilattice valuation (see Definition \ref{def-semilattice-val}). We refer to $\fv$ as a \emph{finite piecewise linear valuation}. We call a piecewise linear valuation {\it integral} if it attains values in $\PL(N, \Z)$. 
\end{definition}


With notation as in Section \ref{sec-tvbs-plm}, to a piecewise linear map $\Phi: |\Sigma| \to \tilde{\B}(E)$ there naturally corresponds a map $\fv_\Phi: E \to \PL(N_\R, \R)$ given by:
\begin{equation}   \label{equ-h-Phi} 
\fv_\Phi(e)(x) = \Phi(x)(e), \quad \forall 0 \neq e \in E, \forall x \in N_\R.  
\end{equation}
We note that the piecewise linear functions $\fv_\Phi(e)$, $0 \neq e \in E$, are not necessarily linear on the cones of $\Sigma$, instead they are piecewise linear with respect a subdivision of $\Sigma$ which depends on the multiset $u(\sigma)$. This is because for $x \in \sigma$ we have $\fv_\Phi(e)(x) = \Phi(x)(e) = \min\{ \langle x, u_{\sigma, i} \rangle \mid e_i \neq 0\}$, where $e = \sum_i e_i$ is the decomposition of $e$ in the frame $L_\sigma$. 

Conversely, suppose $\fv: E \to \PL(N_\R, \R)$ is a piecewise linear valuation. To $\fv$ we can associate a map $\Phi_\fv: |\Sigma| \to \tilde{\B}(E)$ by:
\begin{equation} \label{equ-Phi_h}
\Phi_\fv(x)(e) = \fv(e)(x),~ \forall 0 \neq e \in E ~\forall x \in N_\R.
\end{equation}

The next theorem is the main result of this section. Part (2) in the theorem is the key part and is not an immediate corollary of definitions.
\begin{theorem}    \label{th-plm-vs-preval-plf}
With notation as above, we have the following.
\begin{itemize}
\item[(1)] 
The map $\fv_\Phi: E \to \PL(N_\R, \R)$ is a piecewise linear valuation. 
In fact, there is a subdivision $\Sigma'$ of $\Sigma$ such that $\fv_\Phi: E \to \PL(\Sigma', \R)$. Moreover, if $\Phi$ is integral then $\fv_\Phi: E \to \PL(\Sigma', \Z)$. 
\item[(2)] The map $\Phi_\fv$ is well-defined, that is, for any $x \in N_\R$, the function $\Phi_\fv(x)$ is a valuation on $E$. Moreover, there exists a complete fan $\Sigma$ such that $\Phi_\fv: |\Sigma| \to \tilde{\B}(E)$ is a piecewise linear map (in the sense of Definition \ref{def-plm}).
\item[(3)] The maps $\Phi \mapsto \fv_\Phi$ and $\fv \mapsto \Phi_\fv$ give a one-to-one correspondence between the set of maps $\Phi: N_\R \to \tilde{\B}(E)$ which are piecewise linear (with respect to a complete fan) and the set of finite piecewise linear valuations $\fv: E \to \PL(N_\R, \R)$. Moreover, this restricts to give a one-to-one correspondence between the integral finite piecewise linear maps $\Phi: N \to \tilde{\B}_\Z(E)$ and integral piecewise linear valuations $\fv: E \to \PL(N, \Z)$.
\end{itemize}
\end{theorem}
\begin{proof}
(1) For every $x \in N_\R$, its image $\Phi(x)$ is a valuation $E \to \overline{\R}$. The first claim follows from this. To prove the second claim recall that by definition of a piecewise linear map, for each maximal cone $\sigma \in \Sigma$ there exists a multiset $u_\sigma = \{ u_{\sigma, 1}, \ldots, u_{\sigma, r}\}$ and a frame $L_\sigma$ such that, for all $x \in \sigma$, we have $\fv_\Phi(e)(x) = \Phi(x)(e) = \min\{ \langle x, u_{\sigma, i} \rangle \mid e_i \neq 0 \}$, where $e = \sum_i e_i$ is the decomposition of $e$ in the frame $L_\sigma$. In particular, there are only a finite number of possibilities for the piecewise linear functions $\fv_\Phi(e)$, for all $0 \neq e \in E$. Thus we can take $\Sigma'$ to be a refinement of $\Sigma$ such that all the $\fv_\Phi(e)$ are linear on the cones in $\Sigma'$. 
Finally, if $\Phi$ is integral then all the $\fv_\Phi(e)(x) = \Phi(x)(e)$ attain integer values when $x \in N$. This finishes the proof.

(2) The first claim, namely $\Phi_\fv(x)$ is a valuation for any $x \in N_\R$, follows immediately from the assumption that $\fv$ is a valuation with values in piecewise linear functions. It remains to show the existence of a complete fan $\Sigma$ such that $\Phi_\fv$ is a piecewise linear map with respect to $\Sigma$. Let us choose a finite spanning set $\mathcal{B} = \{b_1, \ldots, b_s\} \subset E$ such that $\fv(\mathcal{B})$ coincides with the image of $\fv$ (excluding $\infty$). For each $b \in \mathcal{B}$, let $\Sigma(b)$ be a complete fan such that $\fv(b)$ is piecewise linear with respect to $\Sigma(b)$. Let $\Sigma$ be a common refinement of all the $\Sigma(b)$, $b \in \mathcal{B}$. Now let us further refine the fan $\Sigma$ into a fan $\Sigma'$ according to the inequalities $\fv(b) \leq \fv(b')$ for all $b, b' \in \mathcal{B}$. Take a maximal cone $\sigma \in \Sigma$. Let $\sigma' \subset \sigma$ be a maximal cone in $\Sigma'$. By the construction of the fan $\Sigma'$, on the cone $\sigma'$ the functions $\h(b)$, $b \in \mathcal{B}$, are totally ordered. Without loss of generality, let us assume $\fv(b_1)(x) \leq \cdots \leq \fv(b_s)(x)$, $\forall x \in \sigma'$. Since $\mathcal{B}$ is a spanning set, we get a flag $F_{\sigma'}$ of subspaces in $E$ obtained by taking the span of $\{b_1, \ldots, b_i\}$ for every $i$. We can then choose a vector space basis $B_{\sigma'} \subset \mathcal{B}$ for $E$ that is adapted to this flag. Note that, by the construction of the fan $\Sigma$, the functions $\fv(b_i)$ are linear on the cone $\sigma$. It follows that if $\sigma'' \subset \sigma$ is another maximal cone in $\Sigma'$, the basis $B_{\sigma'}$ is also adapted to the flag $F_{\sigma''}$ of subspaces corresponding to $\sigma''$. This then implies that the map $\Phi_\fv$ is a piecewise linear map with respect to the fan $\Sigma$. More precisely, let $0 \neq e \in E$ and take $x \in \sigma$. Suppose $x$ lies in a maximal cone $\sigma' \in \Sigma'$ and $\fv(b_1)(x) \leq \cdots \leq \fv(b_s)(x)$. Then since $\fv(\mathcal{B})$ coincides with the image of $\fv$, we should have $\fv(e)(x) = \fv(b_k)(x)$, for some $b_k \in B_{\sigma'}$. On the other hand, since the basis $B_{\sigma'}$ is adapted to the flag $F_{\sigma'}$, if $e = \sum_i e_i$ then $e_i = 0$ for $i \geq k$. It follows that $\fv(e)(x) = \min\{ \fv(b_i)(x) \mid e_i \neq 0 \}$. This shows that $\fv(e)$ is linear on the cone $\sigma$.  

(3) Let $\Phi : |\Sigma| \to \tilde{\B}(E)$ be a piecewise linear map and put $\fv = \fv_\Phi$. It follows from the definition that for any $0 \neq e \in E$ and $x \in N_\R$ we have $\Phi_\fv(x)(e) = \fv(e)(x) = \Phi(x)(e)$ which shows that $\Phi_\fv = \Phi$. Conversely, let $\fv: E \to \PL(N_\R, \R)$ be a valuation and put $\Phi = \Phi_\fv$. In a similar way one verifies that $\fv_\Phi = \fv$. 
To prove the last claim one observes that if $\Phi$ is an integral piecewise linear map then $\fv_\Phi$ is an integral valuation and conversely if $\fv$ is an integral valuation then $\fv_\Phi$ is an integral piecewise linear map.
\end{proof}

Because of the duality between the set of convex polytopes and concave piecewise linear functions, it is natural also to look at valuations with values in convex polytopes. Let $\P(M_\R)$ denote the collection of all convex polytopes in the $\R$-vector space $M_\R$. Partially order $\P(M_\R)$ by reverse inclusion. The set $(\P(M_\R), \subseteq)$ has structure of a meet-join lattice. For $\Delta_1, \Delta_2 \in \P(M_\R)$, their join is $\Delta_1 \cap \Delta_2$ and their meet is the convex hull of $\Delta_1 \cup \Delta_2$. We denote it by $\Delta_1 \vee \Delta_2$ (note that we are considering reverse inclusion and hence meet and join are switched). 

\begin{remark}   \label{rem-semring-polytopes}
The set of polytopes $\P(M_\R)$ is moreover equipped with the Minkowski sum of polytopes which together with the convex hull of union makes it a semiring. In this paper we do not address this semiring structure. This semiring structure makes an appearance in the companion paper \cite{KM-PL}.
\end{remark}


Finally, we recall the correspondence between concave piecewise linear functions and convex polytopes. 
A function $\phi: N_\R \to \R$ is {\it concave} if for any $x_1, x_2 \in N_\R$ and $0 \leq t \leq 1$ we have: $$\phi(tx_1 + (1-t)x_2) \geq t\,\phi(x_1) + (1-t)\,\phi(x_2).$$
Note that the set $\CPL(N_\R, \R)$ of concave piecewise linear functions is closed under taking minimum and hence $(\CPL(N_\R, \R), \geq, \min)$ is a semilattice. 
Given a polytope $\Delta \in \P(M_\R)$ one defines its {\it support function} $\phi_\Delta: N_\R \to \R$ by:
\begin{equation}  \label{def-supp-function}
\phi_\Delta(x) = \min\{ \langle x, y \rangle \mid y \in \Delta\}.
\end{equation}
It is well-known that $\phi_\Delta$ is a piecewise linear function with respect to the normal fan of $\Delta$. Conversely, to each piecewise linear function there corresponds a (possibly empty) polytope $\Delta_\phi \in \P(M_\R)$ defined by:
\begin{equation}   \label{def-polytope-of-plf}
\Delta_\phi = \{ y \in M_\R \mid \langle x, y\rangle \geq \phi(x), \forall x \in N_\R \}.
\end{equation}
It is well-known that the maps $\phi \mapsto \Delta_\phi$ and $\Delta \mapsto \phi_\Delta$ give a one-to-one correspondence between the set of polytopes $\P(M_\R)$ and the set of concave piecewise linear functions $\CPL(N_\R, \R)$.

\begin{proposition}   \label{prop-polytopes-iso-CPL}
With notation as above, the maps $\Delta \mapsto \phi_\Delta$ and $\phi \mapsto \Delta_\phi$ give an isomorphism of the semilattices $(\P(M_\R), \subseteq, \vee)$ and $(\CPL(N_\R, \R), \geq, \min)$. 
\end{proposition}
\begin{proof}
Let $\Delta_1, \Delta_2 \in \P(M_\R)$ be polytopes with $\Delta_1 \subset \Delta_2$. It follows from \eqref{def-supp-function} that $\phi_{\Delta_1} \geq \phi_{\Delta_2}$. Similarly, if  $\phi_1, \phi_2 \in \CPL(N_\R, \R)$ with $\phi_1 \geq \phi_2$ then from \eqref{def-polytope-of-plf} we have $\Delta_{\phi_1} \subset \Delta_{\phi_2}$. This proves the isomorphism between $\P(M_\R)$ and $\CPL(N_\R, \R)$ as posets. 
\end{proof}

\subsection{Toric vector bundles as piecewise linear valuations 
}  \label{subsec-toric-vb-preval-plf}
Fix a torus $T$ with lattice of one-parameter subgroups $N \cong \Z^n$. 
We first consider an equivalence relation on the collection of toric vector bundles on $T$-toric varieties. Let $X_\Sigma$, $X_{\Sigma'}$ be complete $T$-toric varieties equipped with toric vector bundles $\E$, $\E'$ respectively. We say that $(X_\Sigma, \E)$ is {\it equivalent} to $(X_{\Sigma'}, \E')$ if there is a complete toric variety $X_{\Sigma''}$ and $T$-equivariant morphisms $F: X_{\Sigma''} \to X_\Sigma$, $F': X_{\Sigma''} \to X_{\Sigma'}$ such that 
$F^*(\E)$ and ${F'}^*(\E')$ are isomorphic as toric vector bundles on $X_{\Sigma''}$.

It is well-known in toric geometry that the equivalence classes of toric line bundles over $T$-toric varieties are in one-to-one correspondence with integral piecewise linear functions on $N_\R$ (see \cite[Chapter 6]{CLS}). The next theorem can be considered as a generalization of this fact to toric vector bundles.
\begin{theorem}   \label{th-semilattice-preval-toric-vb}
The piecewise linear valuations $\fv: E \to \PL(N, \Z)$ are in one-to-one correspondence with the equivalence classes of toric vector bundles on complete toric varieties (with $E$ as the fiber over the identity). 
\end{theorem}
\begin{proof}
Let $\fv: E \to \PL(N, \Z)$ be a peicewise linear valuation. By Theorem \ref{th-plm-vs-preval-plf}(2) there exists a fan $\Sigma$ such that $\Phi = \Phi_\fv$ is piecewise linear with respect to $\Sigma$. Thus $\fv$ gives rise to a toric vector bundle $\E = \E_\Phi$ on the toric variety $X_\Sigma$. Now suppose $\Sigma'$ is another fan such that $\Phi$ is piecewise linear with respect to $\Sigma'$ as well. Take a common refinement $\Sigma''$ of the fans $\Sigma$ and $\Sigma'$. Clearly, $\Phi$ is piecewise linear with respect to $\Sigma''$ also. Let $\E'$, $\E''$ denote the corresponding toric vector bundle on $X_{\Sigma'}$, $X_{\Sigma''}$ respectively. We have birational $T$-equivariant morphisms $F: X_{\Sigma''} \to X_{\Sigma}$, $F': X_{\Sigma''} \to X_{\Sigma'}$. The equivalence of categories part of Klyachko's classification of toric vector bundles (Theorem \ref{th-Klyachko-plm}) implies that 
$\E'' \cong F^*(\E)$ and $\E'' \cong {F'}^*(\E')$. This shows that $(X_\Sigma, \E)$ and $(X_{\Sigma'}, \E')$ are equivalent as required.
\end{proof}

\begin{example}[Tangent bundle of $\mathbb{P}^2$]   \label{ex-TP^2}
In Example \ref{ex-tangent-bundle-Pn} we saw the Klyachko data and the piecewise linear map associated to the tangent bundle of projective space $\mathbb{P}^n$. 
Let us consider the projective plane $\mathbb{P}^2$ and determine the piecewise linear valuation $\fv = \fv_\E: E=\k^2 \to \PL(N, \Z)$ associated to its tangent bundle $\E = T\mathbb{P}^2$. We follow notation from Example \ref{ex-tangent-bundle-Pn} (see Figure \ref{fig-fan-P2}). By Proposition \ref{prop-arrangement-intersect-E-rho}, the arrangement $\A = \A_\fv$ associated to this piecewise linear valuation is the intersection of all the subspaces appearing in the Klyachko filtrations. The arrangement $\A$ consists of subspaces $\{0\}$, $\k^2$ and $\Span(\v_i)$, $i=1, 2, 3$. One computes that the values of $\fv$ at the vectors $\v_i$ are the following piecewise linear functions (below, $x = (x_1, x_2) \in N_\R = \R^2$):
$$\fv(\v_1) = \begin{cases} \min(x_2 - x_1, -x_1) & x \in \sigma_1 \\ x_1 - x_2 & x \in \sigma_2 \\ x_1 & x \in \sigma_3 \end{cases} \quad 
\fv(\v_2) = \begin{cases} x_2 - x_1 & x \in \sigma_1 \\ \min(x_1 - x_2, -x_2) & x \in \sigma_2 \\ x_2 & x \in \sigma_3 \end{cases}$$ 
$$\fv(\v_3) = \begin{cases} -x_1 & x \in \sigma_1 \\ -x_2 & x \in \sigma_2 \\ \min(x_1, x_2) & x \in \sigma_3. \end{cases}$$
One verifies that $\fv(\v_i)$, $i=1, 2, 3$, are indeed strictly concave functions. 
\begin{figure}[ht] 
\includegraphics[width=5cm]{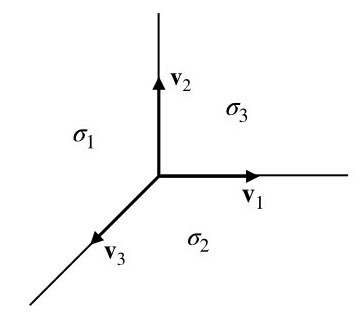} 
\caption{Fan of $\mathbb{P}^2$} \label{fig-fan-P2} 
\end{figure} 
\end{example}

\subsection{Parliament of polytopes of a piecewise linear valuation}  \label{subsec-parliament}
In \cite{DJS} the authors introduce the notion of a \emph{parliament of polytopes} associated to a toric vector bundle over a toric variety $X_\Sigma$. It is a generalization of the Newton polytope/moment polytope of a toric line bundle.
In Section \ref{subsec-preval-plf}, we associated a piecewise linear valuation $\fv: E \to \PL(N, \Z)$ to a toric vector bundle $\E$. In this section, we show how to recover the dimensions of weight spaces of the global sections from $\fv$. To this end, we introduce an extension of the notion of parliament of polytopes. 

Given a toric vector bundle $\E$, its parliament of polytopes $P(\E)$ (in the sense of \cite{DJS}) is a collection of convex polytopes $P_e \subset M_\R$ that are indexed by elements $e$ in the ground set of a matroid $M(\E)$. By abuse of notation, we use $M(\E)$ to denote the ground set of the matroid as well. The matroid $M(\E)$ is the matroid associated to the subspace arrangement in $E$ obtained by intersecting all the subspaces $E_i^\rho$, appearing in the Klyachko filtrations (see Theorem \ref{th-matroid-of-arrangement} and \cite[Paragraph after Proposition 3.1]{DJS}). The parliament is useful in counting the dimensions of the weight spaces of global sections of $\E$. More precisely, one has the following (it is implicit in the proof of \cite[Proposition 1.1]{DJS}):
\begin{proposition}  \label{prop-dim-H^0}
For every character $u \in M$, we have:
$$\dim H^0(X_\Sigma, \E)_u = \rank\{ e \in M(\E) \mid u \in P_e\},$$ where $\rank$ denotes the matroid rank.
\end{proposition}

For a piecewise linear function $\phi \in \PL(N_\R, \R)$ let us define a polytope $P_\phi$ by:
\begin{equation}  \label{equ-P-h}
P_\phi =  \{ y \in M_\R \mid \langle x, y \rangle \leq \phi(x),~\forall x \in N_\R \}.
\end{equation}
We point out that this is the reverse of the inequality \eqref{def-polytope-of-plf}, which was used to relate valuations with values in $\PL(N_\R, \R)$ to valuations with values in $\P(M_\R)$. 

Let $\fv: E \to \PL(N_\R, \R)$ be a finite piecewise linear valuation (Definition \ref{def-PL-val}). Let $S \subset \PL(N, \Z)$ be a subset that is closed under taking maximum. Let $\mathcal{A}_{\fv, S} = \{E_{\geq \phi} \mid \phi \in S \}$ denote the linear subspace arrangement associated to the valuation $\fv$. Also let $M(\fv, S)$ be the matroid associated to $\mathcal{A}_{\fv, S}$ (see Section \ref{subsec-semilattice-val}). By abuse of notation, we denote the ground set of this matroid also by $M(\fv, S)$.

\begin{definition}   \label{def-parliament-PL-val}
We define the \emph{parliament of polytopes} associated to $(\fv, S)$ to be the multiset:
$$P(\fv, S) = \{ P_{\fv(e)} \mid e \in M(\fv, S)\}.$$
\end{definition}

\begin{theorem}  \label{th-parliament-global-sec}
With notation as above, suppose $S$ contains the character lattice $M$. Then, for any $u \in M$, we have the following:
\begin{equation}  \label{equ-parliament}
\dim(H^0(X_\Sigma, \E)_u) = \rank\{e \in M(\fv, 
S) \mid u \in P_{\fv(e)} \}.
\end{equation}
In particular, the smallest choice for $S$ for which \eqref{equ-parliament} holds, is the set consisting of taking all possible maximums of elements in the character lattice $M$.  
\end{theorem}
\begin{proof}
We recall that for a subspace arrangement $\mathcal{A}$ with matroid $M(\mathcal{A})$, we have $\dim(L) = \rank(L \cap M(\mathcal{A}))$, for any $L \in \mathcal{A}$. 
Applying this to the arrangement $\mathcal{A}_{\fv, S}$ and its matroid $M(\fv, S)$, for any $u \in M$ we have:
$$\dim(E_{\fv \geq u}) = \rank\{e \in M(\fv, S) \mid \fv(e) \geq u \} = \{e \in M(\fv, S) \mid u \in P_{\fv(e)}\}.$$
It remains to show that $\dim(H^0(X_\Sigma, \E)_u) = \dim(E_{\fv \geq u})$. 
We note that, by arguments in the Klyachko classification, we have: $$\dim(H^0(X_\Sigma, \E)_u) = \dim(\bigcap_{\rho \in \Sigma(1)} E^\rho_{\langle u, \v_\rho \rangle}) = \dim(\{ e \in E \mid \fv(e)(\v_\rho) \geq \langle u, \v_\rho \rangle, \forall \rho \in \Sigma(1)\}).$$ 
Thus, it is enough to show that for fixed $e \neq 0$, the condition that: $$\fv(e)(\v_\rho) \geq \langle u, \v_\rho \rangle, \quad \forall \rho \in \Sigma(1),$$ implies 
$$\fv(e)(x) \geq \langle u, x\rangle, \quad \forall x \in N_\R.$$ To show this, we note that, for any cone $\sigma \in \Sigma$, any $x \in \sigma$ can be written as $x = \sum_{\rho \in \sigma(1)} c_\rho \v_\rho$ with $c_\rho \geq 0$, and moreover, $\fv(e)(x)$ is given bey $\fv(e)(x) = \min\{ \langle u_{\sigma, i}, x \rangle \mid e_i \neq 0 \}$. Hence:
\begin{align*}
\fv(e)(x) - \langle u, x \rangle &= \min\{\langle u_{\sigma, i} - u, x \rangle \mid e_i \neq 0 \}, \\
&\geq \sum_{\rho \in \sigma(1)} c_\rho \min\{ \langle u_{\sigma, i} - u, \v_\rho \rangle \mid e_i \neq 0 \}, \\
&\geq 0.  
\end{align*}
The last inequality is because, by assumption, $\fv(e)(\v_\rho) = \min\{ \langle u_{\sigma, i}, \v_\rho \rangle \mid e_i \neq 0 \} \geq \langle u, \v_\rho \rangle$, for all $\rho \in \sigma(1)$.
\end{proof}

The above theorem, shows that while the matorid and parliament depend on the choice of subset $S$, if $S$ is sufficiently large, the counting function $\mu_{\fv, S}: M \to \Z_{\geq 0}$ defined by: $$\mu_{\fv, S}(u) = \rank\{e \in M(\fv, S) \mid u \in P_{\fv(e)} \}$$ does not depend on $S$ and only depends on $\fv$. This is reflected in the fact that under a toric pull-back, the weight spaces of global sections do not change.  


To recover the Di Rocco-Jabbusch-Smith parliament of polytopes using the above construction, it is convenient to enlarge the lattice $\PL(N, \Z)$. 
Namely, we consider the larger lattice of $\widehat{\PL}(N, \overline{\Z})$ consisting of the functions $\phi: N \to \overline{Z}$ that are homogeneous of degree $1$, i.e. $\phi(c x) = c \phi(x)$, for all $c \in \Z$ and $x \in N$. Clearly, $(\widehat{\PL}(N, \overline{\Z}), \geq, \min, \max)$ is a lattice conatining $\PL(N, \Z)$ as a sublattice.

For each ray $\rho \in \Sigma(1)$ and $i \in \Z$, let $\phi_{\rho, i} \in \widehat{\PL}(N, \overline{\Z})$ be defined by $\phi_{\rho, i}(\v_\rho) = i$ and $\phi_{\rho, i}(x) = \infty$ for all $x \notin \rho$. Let $S_\Sigma \subset \widehat{\PL}(N, \overline{\Z})$ be the subset obtained by taking maximums of any collection of the $\phi_{\rho, i}$, $\rho \in \Sigma(1)$, $i \in \Z$.
\begin{proposition} \label{prop-arrangement-intersect-E-rho}
Let $\fv$ be a finite piecewise linear valuation. Let $\E$ be a toric vector bundle on a toric variety $X_\Sigma$ representing the equivalence class of toric vector bundles corresponding to $\fv$. We have the following:
\begin{itemize}
\item[(a)] The subspace arrangement $\mathcal{A}_{\fv, S_\Sigma}$ coincides with the Klyachko arrangement obtained by taking intersections of all subspaces appearing in the Klyachko filtrations. 
\item[(b)] The parliament of polytopes $P(\E)$ (as defined in \cite{DJS}) coincides with the parliament of polytopes $P(\fv, S_\Sigma)$. 
\end{itemize}
\end{proposition}
\begin{proof}
It is a straightforward consequence of definitions and constructions. 
\end{proof}

\section{Toric vector bundles as tropical points} \label{sec-tvbs-trop-points}
In this section we show that a toric vector bundle over a toric variety $X_\Sigma$ can be defined by the data of a \emph{tropical point of a linear ideal over the piecewise linear semifield}.  To make this precise, we require the notion of tropicalization over an \emph{idempotent semifield}. 

\begin{definition}[Idempotent semifield]  \label{def-idempotent-semifield}
An \emph{idempotent semifield} $\O$ is a set equipped with commutative and associative operations $\oplus$ and $\otimes$ such that:
\begin{itemize}
\item[(i)] $\otimes$ distributes over $\oplus$,
\item[(ii)] there is a neutral element $\infty \in \O$ with respect to $\oplus$,
\item[(iii)] there is a neutral element $0 \in \O$ with respect to $\otimes$,
\item[(iv)] any element not equal to $\infty$ has an inverse with respect to $\otimes$,
\item[(v)] for any element $a$ we have $a \oplus a = a$. 
\end{itemize}
\end{definition}

The condition (v) above ensures that we can define a partial order on any idempotent semifield, in particular we say that $a \preceq b$ if $a \oplus b = a$. An idempotent semifield $\O$ possesses enough structure to define ``tropical geometry with coefficients in $\O$".

\begin{definition}[Tropical variety over a semifield $\O$] \label{def-tropicalO}
Let $f = \sum_{\alpha \in S} C_\alpha x^\alpha \in \k[x_1, \ldots, x_n]$ be a polynomial, where $S$ is the set of exponents $\alpha$ with $C_\alpha \neq 0$. The \emph{tropicalization} of $f$ over $\O$ is the function $\trop_\O(f): \O^n \to \O$ computed as follows:

\[\trop(f)_\O(a_1, \ldots, a_n) = \bigoplus_{\alpha \in S} a_1^{\otimes \alpha_1}\otimes \cdots \otimes a_n^{\otimes \alpha_n}.\]
The \emph{tropical hypersurface} $\Trop_\O(f)$ is then defined to be the set of $(a_1, \ldots, a_n)$ such that for all $\beta \in S$ we have:

\[\trop_\O(f)(a_1, \ldots, a_n) = \bigoplus_{\alpha \in S \setminus \{\beta\}} a_1^{\otimes \alpha_1}\otimes \cdots \otimes a_n^{\otimes \alpha_n}.\]
In particular, any monomial term of $\trop_\O(f)$ can be dropped without changing the value on a point $(a_1, \ldots, a_n) \in \Trop_\O(f)$. 
(Note that $\trop_\O(f)$ is a function from $\O^n$ to $\O$, while $\Trop_\O(f)$ is a subset of $\O^n$.)
Finally, let $I \subset \k[x_1, \ldots, x_n]$ be a polynomial ideal, the \emph{tropical variety} of $I$ over $\O$ is defined to be $\Trop_\O(
I) = \bigcap_{f \in I} \Trop_\O(f) \subset \O^n$. 

When the semifield is $(\overline{\R}, \min, +)$, we denote the tropical variety of an ideal $I$ simply by $\Trop(I)$.
\end{definition}

Idempotent semifields also have enough structure to allow us to define a notion of a valuation (cf. Definition \ref{def-semilattice-val}). 
\begin{definition}[Semifield valuation] \label{def-valuationO}
Let $A$ be a commutative $\k$-algebra, then a \emph{valuation over $\k$ with values in} $\O$ is defined to be a function $\fv: A \to \O$ satisfying the following for all $f, g \in A$:

\begin{enumerate}
\item $\fv(fg) = \fv(f) + \fv(g)$,
\item $\fv(f + g) \geq \fv(f) \oplus \fv(g)$,
\item $\fv(Cf) = \fv(f)$ for any $0 \neq C \in \k$,
\item $\fv(f) = \infty$ if and only if $f=0$.
\end{enumerate}

\end{definition}

The next well-known proposition links Definitions \ref{def-tropicalO} and \ref{def-valuationO} (this observation is sometimes known as Payne's theorem, see \cite{Payne-analytification, Giansiracusa}). 

\begin{proposition} \label{thm-payne}
Let $F: \k[x_1, \ldots, x_n] \to A$ be a presentation of a $\k$-algebra $A$ with $I = \ker(F)$, and let $f_i = F(x_i)$ for $1 \leq i \leq n$. Then we have $(\fv(f_1), \ldots, \fv(f_n)) \in \Trop_\O(I)$. 
\end{proposition}

The application of these ideas to toric vector bundles starts with the observation that $(\PL(N, \Z), \min, +)$ is a semifield (see Remark \ref{rem-PL-semialgebra}). 
Let $\fv: E \to \PL(N, \Z)$ be a piecewise linear valuation as in Section \ref{sec-tvbs-PL-val}. It is straightforward to see that $\fv$ extends to an algebra valuation (which we denote by the same letter) $\fv: \Sym(E) \to \PL(N, \Z)$, as follows: for $e_1 \cdots e_k \in \Sym^k E$, where $e_1, \ldots, e_k \in E$, define $\fv(e_1 \cdots e_k) = \fv(e_1) + \cdots + \fv(e_k)$. One verifies that this gives a well-defined valuation. 

Let $\mathcal{B} = \{b_1, \ldots, b_s\} \subset E$ be a spanning set. Let $L \subset \k[x_1, \ldots, x_s]$ be the linear ideal generated by the linear relations among the $b_i$.

\begin{theorem}  \label{th-val-vs-trop-pt}
Let $(\phi_1, \ldots, \phi_s) \in \PL(N, \Z)^s$.
Then there exists a piecewise linear valuation $\fv: E \to \PL(N, \Z)$ with $\fv(b_i) = \phi_i$, for all $i$, if and only if $(\phi_1, \ldots, \phi_s) \in \Trop_{\PL(N, \Z)}(L)$. Moreover, $\fv$ is unique, whenever it exists. 
\end{theorem}

In light of Theorem \ref{th-semilattice-preval-toric-vb}, the above shows that tropical points correspond to (equivalence classes of) toric vector bundles (up to pull-back by toric blowups).


Before giving the proof of Theorem \ref{th-val-vs-trop-pt}, we need to recall some basic facts about Gr\"obner fans and tropical varieties. Let $I \subset \k[x_1, \ldots, x_s]$ be an ideal with affine variety $V = V(I)$. The \emph{Berkovich analytification} of $V$, denoted by $V^\text{an}$ is the space of all $\R$-valued valuations on the coordinate ring $\k[V] = \k[x_1, \ldots, x_s] / I$. There is a natural embedding $j: \Trop(I) \hookrightarrow V^\text{an}$ as follows (see \cite{Payne-analytification}): let $a=(a_1, \ldots, a_s) \in \Z^s$ be an integer point in the tropical variety $\Trop(I)$. By the fundamental theorem of tropical geometry, there is a formal curve $\gamma = (\gamma_1, \ldots, \gamma_s)$, where each $\gamma_i \in \k((t))$ is a formal Laurent series in an indeterminate $t$, such that $\gamma$ is a point of $V$ over the field $\k((t))$ and moreover, for every $i$ we have $a_i = \val_t(\gamma_i)$. Here $\val_t: \k((t)) \to \overline{\Z}$ denotes the $t$-adic valuation (or order of $t$ valuation) on the field of Laurent series $\k((t))$. Now, the valuation $j(a): \k[V] \to \Z$ is defined by: 
$$j(a)(f) = \val_t(f(\gamma)), \quad \forall f \in \k[V].$$
The definition of $j$ easily extends to the rational points in $\Trop(L)$ by replacing the field $\k((t))$ with the field of Puiseux series $\k\{\{t\}\}$. One then defines $j$ on the real points by continuity. 

With notation as before, in the case of a linear ideal $L$, the variety $V$ is a linear subspace of $\k^s$ and its coordinate ring can be identified with the symmetric algebra $\Sym(E)$. In this case, the valuation $j(a)$ can be regarded as a vector space valuation on $E$, and hence we have an embedding $j: \Trop(L) \hookrightarrow \tilde{\B}(E)$. We need the following two well-known facts:
\begin{itemize}
\item[(i)] The maximal cones $\tau$ in the Gr\"obner fan of $L$ are in one-to-one correspondence with the subsets $B_\tau$ of $\mathcal{B}$ that are vector space bases for $E$.
\item[(ii)] For every maximal cone $\tau$ in the Gr\"obner fan of $L$ and any point $a \in \Trop(L) \cap \tau$, the valuation $j(a): E \to \overline{\R}$ is adapted to the basis $B_\tau$.
\end{itemize}

\begin{proof}[Proof of Theorem \ref{th-val-vs-trop-pt}]
It follows from Proposition \ref{thm-payne} that if $\fv: E \to \PL(N, \Z)$ is a piecewise linear valuation then $(\fv(b_1), \ldots, \fv(b_s)) \in \Trop_{\PL(N, \Z)}(L)$. So we need to prove the other direction. Suppose $\phi=(\phi_1, \ldots, \phi_s) \in \Trop_{\PL(N, \Z)}(L)$. By evaluating $\phi$ at the points of $N_\R$ we get a map $\psi: N_\R \to \R^s$. Since $\phi$ lies in $\Trop_{\PL(N, \Z)}(L)$, the image of $\psi$ lands in $\Trop(L)$. Now, composing $\psi$ with the embedding $j: \Trop(L) \hookrightarrow \tilde{\B}(E)$ we obtain a map $\Phi: N_\R \to \tilde{\B}(E)$. By Theorem \ref{th-plm-vs-preval-plf}, it suffices to show that there is a complete fan $\Sigma$ such that $\Phi$ is a piecewise linear map with respect to $\Sigma$.
Consider first the fan $\Sigma'$ that is the common refinement of the domains of linearity of the $\phi_i$. Now for each $\sigma' \in \Sigma'$ we further subdivide it by intersecting with the preimages of the faces of the Gr\"obner fan of $L$. This produces a complete fan $\Sigma \subset N_\R$ with the property that each face $\sigma \in \Sigma$ is mapped linearly into an apartment of $\tilde{\B}(E)$.
\end{proof}


Let $\E$ be a toric vector bundle on a complete toric variety $X_\Sigma$ with associated piecewise linear map $\Phi: |\Sigma| \to \tilde{\B}(E)$ and piecewise linear valuation $\fv = \fv_\Phi: E \to \PL(N, \Z)$. Given a spanning set $\mathcal{B} = \{b_1, \ldots, b_s\}$ for $E$, one can ask: when does the values of $\fv$ on $\mathcal{B}$ determine the toric vector bundle $\E$? The following proposition answers this question.

\begin{proposition}\label{prop-tropdata}
Let $\E$ be a toric vector bundle over $X_\Sigma$ with corresponding piecewise linear map $\Phi: |\Sigma| \to \tilde{\B}(E)$. Let $\mathcal{B}=\{b_1, \ldots, b_s\} \subset E$ be a spanning set such that for each face $\sigma \in \Sigma$ there is an apartment of $\tilde{\B}(E)$ containing $\Phi(\sigma)$ with basis $B_\sigma$ a subset of $\mathcal{B}$.  Then the tuple $(\fv(b_1), \ldots, \fv(b_s)) \in \Trop_{\PL(N, \Z)}(L)$ as above determines $\E$. 
\end{proposition}
\begin{proof}
The Klyachko data of $\E$ is recovered from $(\fv(b_1), \ldots, \fv(b_s))$ as follows.  For a facet $\sigma \in \Sigma$, the restriction of each $\fv(b_i)$ for $b_i \in B_\sigma$ is the integral linear function corresponding to a torus character $m_i \in M$; this is the data of a framing at each maximal cone in $\Sigma$. Now fix a ray $\rho \in \Sigma(1)$ and evaluate each entry of $(\fv(b_1), \ldots, \fv(b_s))$ at the integral generator of $\rho$. By definition we obtain a point in $\Trop(L)$, which gives a decreasing $\Z$-filtration of $E$. This combination of a filtration for each ray and a frame for each facet gives the Klyachko data of $\E$.
\end{proof}

Let $\phi = (\phi_1, \ldots, \phi_s)$ be a point in $\Trop_{\PL(N, \Z)}(L)$ and suppose the conditions in Proposition \ref{prop-tropdata} are satisfied. Then each piecewise linear function $\phi_i$ is determined by its values on the rays of $\Sigma$. Let $D$ be the integer $m \times s$ matrix, where $m = |\Sigma(1)|$ is the number of rays, whose $i$-th column is the values of $\phi$ on the primitive vectors on the rays in $\Sigma$. By Proposition \ref{prop-tropdata}, the data of the linear ideal $L$ and the matrix $D$ determines the toric vector bundle $\E$. We call $D$ the \emph{diagram} of $\E$  (with respect to the choice of the spanning set $\mathcal{B}$), see \cite[Section 4]{KM-PL}.

\begin{example}[Tangent bundle of $\mathbb{P}^2$]
We consider the tangent bundle $T\mathbb{P}^2$. In this case the linear ideal is $L = \langle x + y + z \rangle \subset \k[x, y, z]$, and the diagram is the $3 \times 3$ identity matrix.  More generally, the ideal $L = \langle x_0 + \cdots + x_n\rangle$ and a diagonal matrix with all non-negative entries defines an irreducible bundle over $\mathbb{P}^n$.  These bundles were first studied by Kaneyama \cite{Kaneyama}, where he shows that up to tensoring with line bundles, any irreducible bundle of rank $n$ on $\mathbb{P}^n$ is either of this type, or the dual of this type. The Cox rings of the projectivizations of these bundles are studied in \cite{George-Manon}. 
\end{example}

\end{document}